\documentclass[12pt]{article}

\usepackage[hmargin=1.5cm,vmargin=2cm]{geometry}
\usepackage{mathrsfs}

\usepackage{multirow}
\usepackage{t1enc}
\usepackage{pict2e}
\usepackage[latin1]{inputenc}
\usepackage[english]{babel}

\usepackage{float}
\usepackage{amsmath,amsthm, amssymb}
\usepackage{enumerate}
\usepackage{enumitem}
\usepackage{amsfonts}
\usepackage{latexsym,lineno}
\usepackage{listings}
\usepackage{hyperref}

\usepackage{authblk}

\usepackage{tikz}
\usetikzlibrary{calc,decorations.pathmorphing,decorations.text,fixedpointarithmetic}

\theoremstyle{plain}
\newtheorem{thm}{Thm}[section]

\newtheorem{claim}[thm]{Claim}
\newtheorem{theorem}[thm]{Theorem}
\newtheorem{lemma}[thm]{Lemma}
\newtheorem{corollary}[thm]{Corollary}
\newtheorem{proposition}[thm]{Proposition}
\newtheorem{conjecture}[thm]{Conjecture}
\newtheorem{problem}[thm]{Problem}

\newtheorem{definition}[thm]{Definition}

\newtheorem{example}[thm]{Example}

\newenvironment{proof*}{\emph{Proof of the claim:}}{\hfill$\Diamond$}

\newcommand{\floor}[1]{\lfloor #1 \rfloor}

\renewcommand{\pod}[1]{\allowbreak\mathchoice
	{\if@display \mkern 0mu\else \mkern 0mu\fi (#1)}
	{\if@display \mkern 0mu\else \mkern 0mu\fi (#1)}
	{\mkern 1mu(\mathrm{mod}\mkern 4mu #1)}
	{\mkern 0mu(#1)}
}

\usepackage{color}

\newcommand{\overbar}[1]{\mkern 1.5mu\overline{\mkern-1.5mu#1\mkern-1.5mu}\mkern 1.5mu}
\newcommand{\spto}{\ensuremath{\stackrel{s.p.}{\longrightarrow}}}
\newcommand{\swto}{\ensuremath{\stackrel{\tiny switch}{\longrightarrow}}}

\title{Circular chromatic number of signed graphs}

\begin{document}
	\baselineskip 0.65cm

	\author[1]{Reza Naserasr}
	\author[1]{Zhouningxin Wang}
	\author[2]{Xuding Zhu}
	
	\affil[1]{Universit\'{e} de Paris, CNRS, IRIF, F-75006, Paris, France.  
 \newline	 E-mail: \{reza, wangzhou4\}@irif.fr.}
	\affil[2]{Departments of Mathematics, Zhejiang Normal University, Jinhua, 321004, China. E-mail: xdzhu@zjnu.edu.cn. }

\maketitle

\begin{abstract}
A signed graph is a pair $(G, \sigma)$, where $G$ is a graph and $\sigma: E(G) \to \{+, -\}$ is a signature which assigns to each edge of $G$ a sign.  
Various notions of coloring of signed graphs have been studied. In this paper, we extend circular coloring of graphs to signed graphs. Given a signed graph $(G, \sigma)$  a circular $r$-coloring of $(G, \sigma)$ is an assignment $\psi$ of points of a circle of circumference $r$ to the vertices of $G$ such that for every edge $e=uv$ of $G$, if $\sigma(e)=+$, then $\psi(u)$ and $\psi(v)$ have distance at least $1$, and if $\sigma(e)=-$, then $\psi(v)$ and the antipodal of $\psi(u)$ have distance at least $1$. The circular chromatic number $\chi_c(G, \sigma)$ of a signed graph $(G, \sigma)$ is the infimum of those $r$ for which $(G, \sigma)$ admits a circular $r$-coloring. For a graph $G$, we define the signed circular chromatic number of $G$ to be $\max\{\chi_c(G, \sigma): \sigma \text{ is a signature of $G$}\}$. 

We study basic properties of circular coloring of signed graphs and develop tools for calculating $\chi_c(G, \sigma)$. We explore the relation between the circular chromatic number and the signed circular chromatic number of graphs, and present bounds for the signed circular chromatic number of some families of graphs. 
In particular,  we determine the supremum of the signed circular chromatic number of $k$-chromatic graphs of large girth, of simple bipartite planar graphs, $d$-degenerate graphs, simple outerplanar graphs and series-parallel graphs. We construct a signed planar simple graph whose circular chromatic number is $4+\frac{2}{3}$. This is based and improves on a signed graph built by Kardos and Narboni as a counterexample to a conjecture of M\'{a}\v{c}ajov\'{a}, Raspaud, and \v{S}koviera. 
\end{abstract}

\section{Introduction}\label{Introduction}

Assume $r \ge 1$ is a real number. We denote by  $C^r$ the circle of circumference $r$, obtained from the interval $[0,r]$ by identifying $0$ and $r$. Points in $C^r$ are real numbers from $[0, r)$. 
For two points $x, y$ on $C^r$, the {\rm distance} between $x$ and $y$ on $C^r$, denoted by  $d_{\pod{r}}(x,y)$, is the length of the shorter arc of $C^r$ connecting $x$ and $y$. 
Given two real numbers $a$ and $b$, the interval $[a,b]$ on $C^r$ is a closed interval of $C^r$ in clockwise orientation of the circle whose first point is $a\pmod r$ and whose end point is $b\pmod{r}$. For example if $r >4$, then $[4,1] = \{t: 4 \le t < r, \text{ or } 0 \le t \le 1\}$. Intervals $[a,b)$, $(a,b]$ and $(a,b)$ are defined similarly. The length of the interval $[a,b]$ is denoted by $\ell([a,b])$. Thus  $d_{\pod{r}}(x,y) = \min \{\ell( [x,y]), \ell([y,x] )\}$.

Given a graph $G$, a \emph{circular $r$-coloring} of $G$ is a mapping $f: V(G) \to C^r$  such that for any edge $uv \in E(G)$, $d_{\pod{r}}(f(u), f(v)) \ge 1$.  The {\em circular chromatic number} of $G$ is defined as $$\chi_c(G) = \inf\{r: G \text{ admits a circular $r$-coloring}\}.$$

The concept of circular coloring of graphs was introduced by Vince in 1988 in \cite{V88}, where a  different definition was given and the parameter was called the ``star chromatic number''.
The term ``circular chromatic number'' was coined in \cite{Z01} and the above definition was given in \cite{Zhu92}. One important feature of the circular chromatic number is that for any graph $G$, $\chi(G)-1 < \chi_c(G) \le \chi(G)$ and hence $\chi(G) = \lceil \chi_c(G) \rceil$. In this sense, the invariant $\chi_c(G)$ is a refinement of $\chi(G)$ and it contains more information about the structure of $G$. The circular chromatic number of graphs has been studied extensively in the literature, and the reader is referred to \cite{Z01, Z06} for  surveys on this subject.

A \emph{signed graph} is a graph $G=(V, E)$ (allowing loops and multi-edges) together with an assignment $\sigma: E \rightarrow \{+, -\}$, denoted $(G, \sigma)$. An edge with sign $-$ is a \emph{negative edge} and an edge with sign $+$ is a \emph{positive edge}. If $(G, \sigma)$ is a signed graph in which all the edges are positive (respectively, negative), then $(G, \sigma)$ is denoted as $(G, +)$ (respectively, $(G, -)$). When the signature is clear from the context, we may omit the signature and denote the signed graph by $\tilde{G}$.

 In this paper, we extend the concept of circular coloring of graphs to signed graphs. We remark that an extension of circular coloring to signed graphs was also introduced in \cite{KS18}. However, the extension defined in this paper is different. The difference between these two extensions is further discussed in Section \ref{sec:Remarks}.

For each point $x$ on $C^r$, the unique point of distance $\frac{r}{2}$ from $x$ is called the {\em antipodal} of $x$ and is denoted by $\bar{x}$. Given a set $A$ of points on $C^r$, the {\em antipodal} of $A$, denotes by $\bar{A}$, is the set of antipodals of points in $A$.

\begin{definition}
	Given a signed graph $(G, \sigma)$ with no positive loop and a real number $r$, a \emph{circular $r$-coloring} of $(G, \sigma)$ is a mapping $f: V(G) \to C^r$ such that for each positive edge $e=uv$ of $(G, \sigma)$,  $$d_{\pod{r}}(f(u),  f(v)) \ge 1,$$ and for each negative edge $e=uv$ of $(G, \sigma)$, $$d_{\pod{r}}(f(u),  \overbar{f(v)}) \ge 1.$$
	 The circular chromatic number of $(G, \sigma)$ is defined as 
	$$\chi_c(G, \sigma) = \inf \{r \ge 1: G \text{ admits a circular $r$-coloring}\}.$$
\end{definition}

Note that if $e=uv$ is a negative edge, the condition $d_{\pod{r}}(f(u), \overbar{f(v)}) \ge 1$ is equivalent to $d_{\pod{r}}(f(u), f(v)) \le \frac{r}{2}-1$. This definition can be equivalently viewed as an assignment $\varphi$ of intervals of length $1$ (whose center is determined by $f$) to the vertices such that for a positive edge $uv$, the intervals $\varphi(u)$ and $\varphi(v)$ do not intersect and for a negative edge $uv$, the intervals $\overbar{\varphi(u)}$ and $\varphi(v)$ do not intersect. 

Observe that if $(G, \sigma)$ has no edge, then $\chi_c(G, \sigma)=1$, and if $(G, \sigma)$ has an edge, either positive or negative, then $(G, \sigma)$ is not circular $r$-colorable for $r < 2$. As graphs with no edge are not interesting, in the remainder of the paper, we always assume that $r \ge 2$.

It follows from the definition that for any graph $G$,  $\chi_c(G, +)=\chi_c(G)$. So the circular chromatic number of a signed graph is indeed a generalization of the circular chromatic number of a graph.  

\begin{definition}
\label{def-signedchic}
For a simple graph $G$, the {\em signed circular chromatic number} $\chi_c^s(G)$ of $G$ is defined as 
$$\chi_c^s(G) = \max \{\chi_c(G, \sigma): \sigma \text{ is a signature of $G$}\}.$$
\end{definition}

The circular chromatic number of a graph is a refinement of its chromatic number: for any positive integer $k$, a graph $G$ is circular $k$-colorable if and only if $G$ is $k$-colorable. The same is also true for the chromatic number of signed graphs defined based on the notion of $0$-free coloring define by Zaslavsky \cite{Za82}.

\begin{definition}
	Given a signed graph $(G, \sigma)$ and a positive integer $k$. A \emph{$0$-free $2k$-coloring} of $(G, \sigma)$ is a mapping $f: V(G) \to \{\pm 1, \pm 2, \ldots, \pm k\}$ such that for any edge $e=uv$ of $(G, \sigma)$, $f(u) \ne \sigma(e) f(v)$. 
\end{definition}

\begin{proposition}
	Assume $(G, \sigma)$ is a signed graph and $k$ is a positive integer. Then $(G, \sigma)$ is $0$-free $2k$-colorable if and only if $(G, \sigma)$ is circular $2k$-colorable.
\end{proposition}
\begin{proof}
	Assume $f: V(G) \to \{\pm 1, \pm 2, \ldots, \pm k\}$ is any mapping.  Let 
	\[
	g(v) = \begin{cases} f(v)-1, &\text{ if $f(v) \in \{1,2,\ldots, k\}$} \cr
	f(v)+k-1, &\text{ if $f(v) \in \{-1,-2,\ldots, -k\}$}. \cr
	\end{cases}
	\]
	It is straightforward to verify that $g$ is a circular $2k$-coloring of $(G, \sigma)$ if and only if $f$ is a $0$-free $2k$-coloring of $(G, \sigma)$.
\end{proof}

The number of colors used in the 0-free coloring is always even. There have been several attempts to introduce an analogue coloring which uses an odd number of colors. The term 0-free indeed identifies this coloring from a similar coloring where 0 is added to the set of colors and the set of vertices colored with 0 induces an independent set. To be precise, a  $(2k+1)$-coloring of a signed graph uses colors $\{0,\pm 1, \ldots, \pm k\}$, and the constraint is still the same: for any edge $e=uv$ of $G$, $f(u) \ne \sigma(e) f(v)$. In a $(2k+1)$-coloring of a signed graph, the color $0$ is different from the other colors. The antipodal of $0$ is $0$ itself. The set of vertices of color $0$ is an independent set of $G$, and for every other color $i$, vertices colored by color $i$ may be joined by negative edges. In some sense, circular coloring of signed graph provides a more natural generalization of 0-free coloring to colorings of signed graphs with an odd number of colors, where the colors are symmetric.  

In this paper, we shall study basic properties of circular coloring of signed graphs. We shall explore the relation between the circular chromatic number and the signed circular chromatic number of graphs, and prove that for any graph $G$, $\chi_c(G) \le \chi_c^s(G) \le 2 \chi_c(G)$. We prove that the upper bound is tight even when restricted to graphs of arbitrary large girth or bipartite planar graphs. Furthermore, we construct a signed planar simple graph whose circular chromatic number is $4+\frac{2}{3}$.   M\'{a}\v{c}ajov\'{a}, Raspaud, and \v{S}koviera \cite{MRS16} conjectured that every signed planar simple graph is $4$-colorable. By Proposition \ref{prop-0free}, this is equivalent to say that  $\chi_c^s(G) \le 4$ for every planar graph.  Kardos and Narboni \cite{KN20} refuted this conjecture by constructing a non-4-colorable signed planar graph. Our construction improves on the example of  Kardos and Narboni. Thus we show that the supremum of the signed chromatic number of planar graphs is between $4+\frac 23$ and $6$. The exact value remains an open problem.

\section{Equivalent definitions}

There are several equivalent definitions of the circular chromatic number of graphs. Some of these definitions are also extended naturally to signed graphs.

Note that for $s, t \in [0,r)$,  $d_{\pod{r}}(s,t) = \min \{|s-t|, r-|s-t|\}.$ So a circular $r$-coloring of a graph can be defined as follows, which is sometimes more convenient.

\begin{definition}
	A \emph{circular $r$-coloring} of a signed graph $(G, \sigma)$ is a mapping $f: V(G) \to [0,r)$ such that for each positive edge $uv$, $$1 \le |f(u)-f(v)| \le r-1$$ and for each negative edge $uv$, $$\text{ either } |f(u)-f(v)| \le \frac{r}{2} -1 \text{ or } |f(u)-f(v)| \ge \frac{r}{2} +1.$$
\end{definition}

If $r$ is a rational number, then in a circular $r$-coloring of a signed graph $(G, \sigma)$, it suffices to use a finite set of colors from the interval $[0,r)$. We may assume that $r=\frac{p}{q}$, where $p$ is even and subject to this condition $\frac{p}{q}$ is in its simplest form. For $i \in \{0,1,\ldots, p-1\}$, let $I_i$ be the half open, half closed interval $[\frac{i}{q}, \frac{i+1}{q})$ of $[0, r)$. Then $\cup_{i=0}^{p-1}I_i$ is a partition of $[0,r)$. Assume $f: V(G) \to [0,r)$ is a circular $r$-coloring of a signed graph $(G, \sigma)$. Then for each vertex $v$ of $G$, let $g(v)=\frac{i}{q}$ if and only if $f(v) \in  I_i$. If $e=uv$ is a positive edge, then $1 \le |f(u)- f(v)| \le \frac{p}{q} -1$. This implies that $1 - \frac{1}{q} < |g(u)-g(v)| < \frac{p}{q} -1+ \frac{1}{q}$. Since $q|g(u)-g(v)|$ is an integer, we conclude that $1 \le |g(u)-g(v)| \le \frac{p}{q} -1$. If $e=uv$ is a negative edge, then either $|g(u)-g(v)| <  \frac{p}{2} - 1 + \frac{1}{q}$ or $|g(u)-g(v)| > \frac{p}{2} + 1 - \frac{1}{q}$. Since $p$ is even, $\frac{p}{2}$ is an integer. As $q|g(u)-g(v)|$ is an integer, we conclude that either $|g(u)-g(v)| \le  \frac{p}{2} - 1$ or $|g(u)-g(v)| \ge \frac{p}{2} + 1$.
It is crucial that $p$ be an even integer. For otherwise $\frac{p}{2}$ is not an integer, and we cannot conclude that $|g(u)-g(v)| \le  \frac{p}{2} - 1$ or $|g(u)-g(v)| \ge \frac{p}{2} + 1$. Indeed, if $p$ is odd, then the set $\{0,\frac{1}{q},\ldots, \frac{p-1}{q}\}$ is not closed under taking antipodal points. 

The above observation leads to the following equivalent definition of the circular chromatic number of signed graphs. For $i,j \in \{0,1,\ldots, p-1\}$, the modulo-$p$ distance between $i$ and $j$ is $$d_{\pod{p}}(i,j) = \min \{|i-j|, p-|i-j|\}.$$ Given an even integer $p$, the antipodal color of $x \in \{0,1,\ldots, p-1\}$ is $\bar{x} = x+ \frac{p}{2} \pmod{p}$. 

\begin{definition}
	Assume $p$ is an even integer and $q \le p/2$ is a positive integer. A \emph{$(p, q)$-coloring} of a signed graph $(G, \sigma)$ is a mapping $f: V(G) \to \{0,1,\ldots, p-1\}$ such that for any positive edge $uv$, 
	$$d_{\pod{p}}(f(u), f(v)) \ge q ,$$
	and for any negative edge $uv$,
	$$d_{\pod{p}}(f(u), \overbar{f(v)}) \ge q.$$ 
	The \emph{circular chromatic number} of $(G, \sigma)$ is
	$$\chi_c(G, \sigma) = \inf\{\dfrac{p}{q}: p \text{ is an even integer and $(G, \sigma)$ has a $(p, q)$-coloring}\}.$$
\end{definition}

Note that $d_{\pod{p}}(i, j) \ge q$ is equivalent to $$q \le |i-j| \le p-q.$$

A {\em homomorphism} of a graph $G$ to a graph $H$ is a mapping $f: V(G) \to V(H)$ such that for every edge $uv$ of $G$, $f(u)f(v)$ is an edge of $H$. It is well-known and easy to see that a graph $G$ is $k$-colorable if and only if $G$ admits a homomorphism to $K_k$, the complete graph on $k$ vertices. Similarly, circular chromatic number of graphs are also defined through graph homomorphism. For integers $p \ge 2q >0$, the {\em circular clique} $K_{p;q}$ has vertex set $[p] = \{0,1,\ldots, p-1\}$ and edge set $\{ij: q \le |i-j| \le p-q\}$. Then a circular $\frac{p}{q}$-coloring of a graph $G$ is equivalent to a homomorphism of $G$ to $K_{p;q}$. Circular chromatic number of signed graphs can also be defined through homomorphisms.

\begin{definition}
	 An \emph{edge-sign preserving homomorphism} of a signed graph $(G, \sigma)$ to a signed graph $(H, \pi)$ is a mapping $f: V(G) \to V(H)$ such that for every positive (respectively, negative) edge $uv$ of $(G, \sigma)$, $f(u)f(v)$ is a positive (respectively, negative) edge of $(H, \pi)$.
\end{definition}

We write $(G, \sigma) \spto (H, \pi)$ if there exists an edge-sign preserving homomorphism of $(G, \sigma)$ to $(H, \pi)$.

For integers $p \ge 2q > 0$ such that $p$ is even, the {\em signed circular clique} $K_{p;q}^s$ has vertex set $[p] = \{0,1,\ldots, p-1\}$, in which $ij$ is a positive edge if and only if  $q \le |i-j| \le p-q$ and $ij$ is a negative edge if and only if either $|i-j| \le \frac{p}{2}-q$ or $|i-j| \ge \frac{p}{2}+q$. If $q=1$, then $K_{p;1}^s$ is also written as $K_p^s$. 

Note that in $K_{p;q}^s$, each vertex $i$ is incident to a negative loop. When $\frac{p}{q}\geq 4$, there are parallel edges of different signs. Furthermore, the subgraph induced by all the positive edges of $K^s_{p;q}$ is the circular clique $K_{p;q}$, which is known to be of circular chromatic number $\frac{p}{q}$, thus we have $\chi_c(K^s_{p;q})=\frac{p}{q}$.

The following lemma gives another equivalent definition of the circular chromatic number of a signed graph. 

\begin{lemma}
\label{lem-homo}
	Assume $(G, \sigma)$ is a signed graph, $p$ is a positive even integer, $q$ is a positive integer and $p \ge 2q$. Then $(G, \sigma)$ has a $(p,q)$-coloring if and only if $(G, \sigma) \spto K_{p;q}^s$. Hence the  circular chromatic number of $(G, \sigma)$ is $$\chi_c(G, \sigma) = \inf\{\dfrac{p}{q}: (G, \sigma) \spto  K_{p;q}^s\}.$$
\end{lemma}

As homomorphism relation is transitive, we have the following lemma.

\begin{lemma}\label{lem:No-HomByCircularChromatic}
If $(G,\sigma)\spto (H, \pi)$, then we have $\chi_c(G,\sigma) \leq \chi_c(H, \pi)$.
\end{lemma} 

For a real number $r \ge 2$, we can also define $K_r^s$ be the infinite graph with vertex set $[0,r)$, in which $xy$ is a positive edge if $1 \le |x-y| \le r-1$ and $xy$ is a negative edge if either $|x-y| \le \frac{r}{2} -1$ or $|x-y| \ge \frac{r}{2} +1$. 
Then it follows from the definition that a signed graph $(G, \sigma)$ is circular $r$-colorable if and only if $(G,\sigma)$ admits an edge-sign preserving homomorphism to $K_r^s$. If $r=\frac{p}{q}$ is a rational and $p$ is an even integer, then it follows from the definition that $K_{p;q}^s$ is a subgraph of $K_r^s$. On the other hand, it follows from Lemma \ref{lem-homo} that $K_r^s$ admits an edge-sign preserving homomorphism  to $K_{p;q}^s$.
Note that if $r' \ge r$ then $f: [0,r) \to [0, r')$ defined as $f(x) = \frac{r'x}{r}$ is an edge-sign preserving homomorphism  of $K_r^s $ to $ K_{r'}^s$.  

\begin{lemma}
\label{lem-rr'}
Given even positive integers $p, p'$, if $\frac{p}{q} \le \frac{p'}{q'}$, then $K_{p;q}^s \spto K_{p';q'}^s$.
\end{lemma}
\begin{proof}
	Let $r=\frac{p}{q}, r'=\frac{p'}{q'}$. Then $K_{p;q}^s \spto K_r^s \spto K_{r'}^s \spto K_{p';q'}^s$. 
\end{proof}

Assume $(G, \sigma)$ is a signed graph.
A \emph{switching at vertex $v$} is to switch the signs of edges which are incident to $v$.
A \emph{switching at a set $A\subset V(G)$} is to switch at each vertex in $A$. That is equivalent to switching the signs of all edges in the edge-cut $E(A, V(G)\setminus A)$.  A signed graph $(G, \sigma)$ is a \emph{switching} of $(G, \sigma')$ if it is obtained from $(G, \sigma')$ by a sequence of switchings. We say $(G, \sigma)$ is \emph{switching equivalent} to $(G, \sigma')$ if $(G, \sigma)$ is a switching of $(G, \sigma')$. It is easily observed that given a graph $G$, the relation ``switching equivalent'' is an equivalence class on the set of all signatures on $G$.

It was observed in \cite{Z82} that if $(G, \sigma)$ admits a 0-free $2k$-coloring then every switching equivalent signed graph $(G, \sigma')$ admits such a coloring: If $c$ is  a 0-free $2k$-coloring of $(G, \sigma)$, then after a switching at a vertex $v$ one may change the color of $v$ from $c(v)$ to $-c(v)$ to preserve the property of being a 0-free $2k$-coloring. The same argument applies to circular $r$-coloring.

\begin{proposition}
	\label{prop-switch}
	Assume $(G, \sigma)$ and $(G, \sigma')$ are switching equivalent, say $(G, \sigma')$ is obtained from $(G, \sigma)$ by switching at a set $A$. Then every circular $r$-coloring of $(G, \sigma)$ corresponds to a circular $r$-coloring of $(G, \sigma').$ In particular, $\chi_c(G,\sigma')=\chi_c(G,\sigma)$. 
\end{proposition}
\begin{proof}
	Assume $f$ is a $(p,q)$-coloring of $(G, \sigma)$ and $(G, \sigma')$ is obtained from $(G, \sigma)$ by switching at a set $A$. Let $g: V(G) \to \{0,1,\ldots, p-1\}$ be defined as 
	\[
	g(v) = \begin{cases}
	 f(v), & \text{ if $v \in V(G)-A$}, \cr
	 f(v) + \frac{p}{2}, &\text { if $v \in A$}. \cr
	 \end{cases}
	 \]
	 Here the addition $f(v) + p/2$ is carried out modulo $p$, so that  $f(v) + p/2 \in \{0,1,\ldots, p-1\}$.   
	 It is easy to verify that $g$ is a $(p,q)$-coloring of $(G, \sigma')$.
\end{proof}

Assume $(G, \sigma)$ is a signed graph and $c$ is a $(p, q)$-coloring of $(G, \sigma)$ (where $p$ is even and subject to this condition $\frac{p}{q}$ is in its simplest form). Let $A=\{v: c(v) \ge \frac{p}{2}\}$ and let $(G, \sigma')$ be obtained from $(G, \sigma)$ by switching at $A$. It follows from the proof of Proposition \ref{prop-switch} that there is a $(p, q)$-coloring $c'$ of $(G, \sigma')$ such that $c'(v) \le \frac{p}{2}-1$ for each vertex $v$. 
 Let $\hat{K}_{p;q}^s$ be the signed subgraph of $K_{p;q}^s$ induced by vertices $\{0,1,\ldots, \frac{p}{2}-1\}$. 
 
 \begin{definition}
 	\label{def-signhomo}
 	Assume $(G, \sigma)$ and $(H, \pi)$ are signed graphs. If there is a signed graph $(G, \sigma')$ which is switching equivalent to $(G, \sigma)$ and admits an edge-sign preserving homomorphism  to $(H,\pi)$, then we say  $(G, \sigma)$ admits a {\em switching homomorphism} to $(H, \pi)$.  We write $(G, \sigma) \swto (H,\pi)$ if $(G, \sigma)$ admits a switching homomorphism to $(H,\pi)$.  
 \end{definition}
 
Then we have the following lemma, which can be viewed as another definition of circular chromatic number of signed graphs.

 \begin{lemma}
 	\label{lem-pq2}
 	Assume $(G, \sigma)$ is a signed graph. Then 
 	\[\chi_c(G, \sigma) = \inf\{ \dfrac{p}{q} : p \text{ is even and $(G, \sigma) \swto \hat{K}_{p;q}^s$ } \}.\]
 \end{lemma}
 
 Thus, in particular, Lemma~\ref{lem:No-HomByCircularChromatic} and Lemma~\ref{lem-rr'} can be restated with a switching homomorphism in place of edge-sign preserving homomorphism.
 
 Note that in the graph $\hat{K}_{p}^s$, every pair of distinct vertices are joined by a positive edge and a negative edge, and moreover, each vertex $i$ is incident to a negative loop. Thus we have the following result.
 
 \begin{proposition}
 	\label{prop-0free}
 	A signed graph $(G, \sigma)$ is $(2k,1)$-colorable (equivalently $0$-free $2k$-colorable) if and only if there is a set $A$ of vertices such that after switching at $A$, the result is a signed graph whose positive edges induce a $k$-colorable graph.
 \end{proposition}
 
In the study of circular coloring of signed graphs, switching-equivalent signed graphs are viewed as the same signed graph. 
The problem as which signed graphs are equivalent was first studied by Zaslavsky \cite{Z82}. We define the \emph{sign of a cycle} (respectively, a closed walk) in $(G, \sigma)$ to be  the product of the signs of the edges of the cycle (respectively, the closed walk). 
One may observe that a switching does not change the sign of a cycle of $(G, \sigma)$. A result of Zaslavsky, fundamental in the study of signed graphs, shows that a switching equivalent class to which $(G,\sigma)$ belongs to is determined by signs of all cycles of $(G, \sigma)$.

\begin{theorem}\emph{\cite{Z82}}\label{prop:EquivalenceOfSignatures}
	Two signed graphs $(G, \sigma_1)$ and $(G, \sigma_2)$ are switching equivalent if and only if they have the same set of negative cycles.
\end{theorem}

Thus we have the following proposition (see~\cite{NSZ20} for more details).

\begin{proposition}
	\label{prop-homocycle}
	A signed graph $(G, \sigma)$ admits a switching homomorphism to a signed graph $(H, \pi)$ if and only if there is a homomorphism $f$ from $G$ to $H$ such that for every closed walk $W$ of $(G, \sigma)$, $W$ and $f(W)$ have the same sign. 
\end{proposition}

The following lemma follows from Theorem~\ref{prop:EquivalenceOfSignatures}.

\begin{lemma}
A signed graph $(G, \sigma)$ admits a switching homomorphism to $(H, \pi)$ if and only if there is a mapping of vertices and edges of $(G, \sigma)$ to the vertices and edges of $(H, \pi)$ which preserves adjacencies, incidences, and signs of closed walks. 
\end{lemma}

For a non-zero integer $\ell$, we denote by $C_{\ell}$ the cycle of length $|\ell|$ whose sign agrees with the sign of $\ell$. So for example $C_{-4}$ is a negative cycle of length $4$. Observe that the signed graph $\hat{K}_{4k;2k-1}$ is obtained from $C_{-2k}$ by adding a negative loop at each vertex. Note that adding negative loops to a signed graph or deleting them does not affect its circular chromatic number. So we may ignore negative loops in $(G, \sigma)$. However, as a target of switching homomorphism, negative loops are important, because we can map two vertices connected by a negative edge to a same vertex $v$, provided $v$ is incident to a negative loop.

\section{Some basic properties}
 
Assume $(G, \sigma)$ is a signed graph and $\phi: V(G) \to [0,r)$ is a circular $r$-coloring of $(G, \sigma)$. The \emph{partial orientation $D=D_{\phi}(G, \sigma)$} of $G$ with respect to a circular $r$-coloring $\phi$ is defined as follows: $(u,v)$ is an arc of $D$ if and only if one of the following holds: 
\begin{itemize}
	\item  $uv$ is a positive edge and $(\phi(v)-\phi(u)) \pmod{r} = 1$. 
	\item  $uv$ is a negative edge and $(\overbar{\phi(v)}-\phi(u)) \pmod{r} = 1$. 
\end{itemize}
  
 \begin{definition}
 	Assume $(G, \sigma)$ is a signed graph and $\phi$ is a circular $r$-coloring of $(G, \sigma)$. Arcs in $D_\phi(G, \sigma)$ are called {\em tight arcs} of $(G, \sigma)$ with respect to $\phi$. A directed path (respectively, a directed cycle) in $D_\phi(G, \sigma)$ is called a {\em tight path} (respectively, a {\em tight cycle}) with respect to $\phi$.  
 \end{definition}
 
 \begin{lemma}\label{lem1}
 	Let $(G, \sigma)$ be a signed graph and let $\phi$ be a circular $r$-coloring of $(G, \sigma)$. If $D_{\phi}(G, \sigma)$ is acyclic, then there exists an $r_0\lneqq r$ such that $(G, \sigma)$ admits an $r_0$-circular coloring.
 \end{lemma}
 
 \begin{proof}
 	For a given signed graph $(G, \sigma)$ and a circular $r$-coloring $\phi$ of $(G, \sigma)$, suppose that $D_{\phi}{(G, \sigma)}$ is acyclic. Moreover, we assume among all such $\phi$, $D_{\phi}{(G, \sigma)}$ has minimum number of arcs. First we show that $D_{\phi}{(G, \sigma)}$ has no arc.  
 	Otherwise, since $D_{\phi}{(G, \sigma)}$ is acyclic, $D_{\phi}{(G, \sigma)}$ has an arc $(v,u)$ such that $u$ is a sink. Thus for every positive edge $uw$, $(\phi(w) - \phi(u) )  \pmod{r} > 1$ and for every negative edge $uw$, $(\overbar{ \phi(w)} - \phi(u) )  \pmod{r} > 1$. As $G$ is finite, there exists an $\epsilon > 0$ such that for every positive edge $uw$ of $(G, \sigma)$, 	$(\phi(w) - \phi(u) )  \pmod{r} > 1+ \epsilon$ and for every negative edge $uw$, $(\overbar{ \phi(w)} - \phi(u) ) \pmod{r} > 1+ \epsilon$.
 	   
 	   Let $\psi(x) = \phi(x)$ for $x \ne u$ and $\psi(u) = \phi(u)+ \epsilon$. Then $\psi$ is a circular $r$-coloring of $(G, \sigma)$ and $D_{\psi}{(G, \sigma)}$ is a sub digraph of $D_{\phi}{(G, \sigma)}$, in which $(v,u)$ is not an arc. So $D_{\psi}{(G, \sigma)}$ is acyclic and has fewer arcs than $D_{\phi}{(G, \sigma)}$, contrary to our choice of $\phi$.
 	   
 	  As $D_{\phi}{(G, \sigma)}$ has no arc, it follows from the definition that there exists $\epsilon > 0$ such that for any positive edge $uv$,   $$1+\epsilon \le |\phi(u)-\phi(v)| \le r-(1+\epsilon)$$ and for any negative edge $uv$,   $$1+\epsilon \le |\overbar{\phi(u)}-\phi(v)| \le r-(1+\epsilon).$$
 	  Let $r_0 =\frac{r}{1+\epsilon}$ and let $\psi: V(G) \to [0,r')$ be defined as $\psi(v) = \frac{\phi(v)}{1+\epsilon}$. Then $\psi$ is an $r_0$-circular coloring of $ (G, \sigma)$.
 \end{proof}

 \begin{corollary}\label{coro:TightCycle}
 	If $\chi_c(G, \sigma)=r$, then every circular $r$-coloring $\phi$ of $(G, \sigma)$ has a tight cycle.
 \end{corollary}
 
 The converse of Corollary \ref{coro:TightCycle} is also true.
 
\begin{lemma}\label{lem:Circular-TightCycle}
 	Given a signed graph $(G, \sigma)$, $\chi_c(G, \sigma)=r$ if and only if $(G, \sigma)$ is circular $r$-colorable and every circular $r$-coloring $\phi$ of $(G, \sigma)$, has a tight cycle.
\end{lemma}
 
 \begin{proof}
 	One direction is proved in Corollary~\ref{coro:TightCycle}. It remains to show that if $\chi_c(G, \sigma) < r$, then there is a circular $r$-coloring $\phi$ of $(G, \sigma)$ such that   $D_{\phi}(G, \sigma)$ is acyclic.
 	
 	 Assume $\chi_c(G, \sigma) = r' < r$. Let $\psi: V(G) \to [0,r)$ be a circular $r'$-coloring of $(G, \sigma)$. Let $\phi(v) =\frac{r}{r'} \psi(v) $. Then it is easy to verify that $\phi$ is a circular $r$-coloring of $(G, \sigma)$ and $D_{\phi}(G, \sigma)$ contains no arc (and hence is acyclic).   
 \end{proof}

 \begin{proposition}
 \label{prop-st}
 	Any signed graph $(G, \sigma)$ which is not a forest has a cycle with $s$ positive edges and $t$ negative edges such that $\chi_c(G, \sigma)= \frac{2(s+t)}{2a+t}$ for some non-negative integer $a$. 
 \end{proposition}
 
 \begin{proof} 
 Assume $\chi_c(G, \sigma)=r$ and $\psi: V(G) \to [0,r)$ is a circular $r$-coloring of $(G, \sigma)$. By Lemma~\ref{lem:Circular-TightCycle}, $D_{\psi}{(G, \sigma)}$ contains a directed cycle $B$.  Assume $B$  consists of $s$ positive edges and $t$ negative edges. We view the colors as the points of a circle $C^r$ of circumference $r$, which is obtained from the interval $[0,r]$ by identifying $0$ and $r$. Assume $B=(v_1,v_2, \ldots, v_{s+t})$. If $v_iv_{i+1}$ is a positive edge, then traversing from the colors of $v_i$, one unit along the clockwise direction of $C^r$, we arrive at the color of $v_{i+1}$. 
 If $v_iv_{i+1}$ is a negative edge, then from the color of $v_i$, by first traversing $\frac{r}{2}$  unit along the anti-clockwise direction of $C^r$ then traversing along the clockwise direction a unit distance, we arrive at the color of $v_{i+1}$.
Therefore, directed cycle $B$ represents a total traverse along the circle $C^r$ distance $s-(\frac{r}{2}-1)\cdot t$, at end of which one must come back to the starting color. So  $$s-(\dfrac{r}{2}-1)\times t=r\times a$$ for some integer $a$. Hence \[r=\dfrac{2(s+t)}{2a+t}\]\label{formula:TightCycle}
\end{proof}

Since $s+t \le |V(G)|$,  and $r \ge 2$, given the number of vertices of $G$, there is a finite number of candidates for the circular chromatic number of $(G, \sigma)$. Thus we have the following corollary.
 
 \begin{corollary}
 	\label{coro-min}
 	 Assume $(G, \sigma)$ is a signed graph on $n$ vertices. Then $\chi_c(G, \sigma) = \frac{p}{q}$ for some $p \le 2n$. In particular, the infimum in the definition of $\chi_c(G, \sigma)$ can be replaced by minimum.
 \end{corollary} 
 
 It also follows from Corollary \ref{coro-min} that there is an algorithm that determines  the circular circular chromatic number of a finite signed graph. Of course, determining the circular chromatic number of a signed graph is at least as hard as determining the chromatic number of a graph, and, hence, the problem is NP-hard and, unless P=NP, there is no feasible algorithm for the problem. 
 Nevertheless, it is easy to determine whether a signed graph $(G, \sigma)$ has circular chromatic number $2$.
 
 \begin{proposition}\label{prop:Xc>=2}
 	Any  signed graph $(G, \sigma)$ with at least one edge has $\chi_c(G, \sigma)\geq 2$, and $\chi_c(G,\sigma)=2$ if and only if $(G, \sigma)$ is switching equivalent to $(G, -)$.
 \end{proposition}
 
 Assume $\chi_c(G, \sigma) = r$ and $f: V(G) \to [0,r)$ is a circular $r$-coloring of $(G, \sigma)$. Let $A = \{v: f(v) \ge \frac{r}{2}\}$. Let $(G, \sigma')$ be obtained from $(G, \sigma)$ by switching at $A$. Then 
 \[
 g(v) = \begin{cases} f(v), & \text {if $f(v) < \frac{r}{2}$} \cr
 f(v)-\frac{r}{2}, &\text{ if $f(v) \ge \frac{r}{2}$} \cr
 \end{cases}
 \]
 is a circular $r$-coloring of $(G, \sigma')$. A tight cycle $B=(v_1, v_2, \ldots, v_l)$ with respect to $f$ is also a tight cycle with respect to $g$. However, for each edge $(v_i, v_{i+1})$,
 $(g(v_i), g(v_{i+1}))$ is an arc on the circle of length $\frac{r}{2}$ along the clockwise direction. 
 
Recall that the core of a graph $G$ is a smallest subgraph $H$ of $G$ to which $G$ admits a homomorphism. If $(G, \sigma)$ is a signed graph and $H$ is a subgraph of $G$, then we denote by $(H, \sigma)$ the signed subgraph of $(G, \sigma)$, where $\sigma$ in $(H, \sigma)$ is considered to be the restriction of $\sigma$ to $E(H)$.  
We define the \emph{sp-core} of a signed graph $(G, \sigma)$ to be a smallest signed subgraph $(H, \sigma)$ such that $(G, \sigma)$ admits an edge-sign preserving homomorphism  to $(H, \sigma)$.  
The {\em switching core} of a signed graph $(G, \sigma)$ is a smallest signed subgraph $(H, \sigma)$ such that $(G, \sigma)$ admits a  switching homomorphism to $(H, \sigma)$. That the sp-core and the switching core of a finite signed graph is unique up to isomorphism and thus the well-definiteness is shown in \cite{NRS15}. 

It follows from the definition that the switching core of $(G, \sigma)$ is isomorphic to a signed subgraph of a sp-core of $(G, \sigma)$. 

\begin{lemma}
	\label{lem-switchcoreofKr}
	Assume $r=\frac{p}{q}$ is a rational, $p$ is an even integer and with respect to this condition $\frac{p}{q}$ is in its simplest form. Then $\hat{K}_{p;q}^s$ is the unique switching core of $K_r^s$.
\end{lemma}
\begin{proof}
	Since $\hat{K}_{p;q}^s$ is a subgraph of $K_r^s$ and $K_r^s \swto \hat{K}_{p;q}^s$, it suffices to show that $\hat{K}_{p;q}^s$ is a switching core, i.e., it is not switching homomorphic to any of its proper signed subgraphs. 
	
	Assume to the contrary that there is a switching homomorphism of $\hat{K}_{p;q}^s$ to a proper signed subgraph, say $(H, \sigma)$. As $(H, \sigma) \swto \hat{K}^s_{p;q}$ and $ \hat{K}^s_{p;q} \swto (H, \sigma)$, we have $\chi_c(H, \sigma)= \chi_c(\hat{K}^s_{p;q})=\frac{p}{q}$. 
	
	Let $\phi$ be a circular $\frac{p}{q}$-coloring of $(H, \sigma)$. By Corollary~\ref{coro:TightCycle}, there is a tight cycle $C$ with respect to $\phi$. Assume the length of $C$ is $l$. Since $\frac{p}{q}$ is in its simplest form, beside a possible factor of $2$, we consider two cases: $q$ is odd, then $p\mid 2l$, which implies that $l\geq \frac{p}{2}$; or $q$ is even then $\frac{p}{2}$ must be an odd number, thus $\frac{p}{2}\mid l$, hence, $l\geq \frac{p}{2}$. But $C$ is a cycle in $(H, \sigma)$, which is a proper subgraph of $\hat{K}_{p;q}^s$, a contradiction.
\end{proof}

\begin{lemma}
	\label{lem-coreofKr}
	Assume $r=\frac{p}{q}$ is a rational, $p$ is an even integer and with respect to this condition $\frac{p}{q}$ is in its simplest form. Then $K_{p;q}^s$ is the unique sp-core of $K_r^s$.
\end{lemma}

\begin{proof}
As $K^s_r\spto K_{p;q}^s$, it is enough to prove that $K^s_{p;q}$ is a sp-core. 
Let $(H, \sigma)$ be the sp-core of $K^s_{p;q}$ which is a proper subgraph and let $\varphi$ be an edge-sign preserving homomorphism of $K^s_{p;q}$ to $(H, \sigma)$. Since any edge-sign preserving homomorphism is, in particular, a switching homomorphism and by Lemma~\ref{lem-switchcoreofKr}, $\hat{K}^s_{p;q}$ is a subgraph of $(H, \sigma)$. Observe that for each vertex $u$ of $\hat{K}^s_{p;q}$ there are two corresponding vertices $u_1$ and $u_2$ of $K^s_{p;q}$ such that a switching at $u_1$ gives $u_2$. 
Furthermore, there exists a positive edge $u_1u_2$ in $K^s_{p;q}$. So $\varphi(u_1)\neq \varphi(u_2)$. Moreover $\varphi(v_i)\neq \varphi(u_j)$, for any $i,j\in \{1,2\}$ and for any other vertex $v$ of $\hat{K}^s_{p;q}$, as otherwise we have an edge-sign preserving homomorphism  of $\hat{K}^s_{p;q}$ to its proper subgraph by mapping $u$ to $v$. It is a contradiction. 
\end{proof}

 \section{Circular chromatic number vs. signed circular chromatic number}
 
  The following lemma follows from the definitions.
  
 \begin{lemma}
 	\label{lem-circ-signedcirc}
 	For any integers $p,q$, satisfying $p \ge 2q \ge 2$, for any $i,j \in [p]$, $ij$ is an edge of $K_{p;q}$ if and only if $ij$ is both a positive edge and a negative edge of $\hat{K}_{2p;q}^s$.
 \end{lemma}
  
 \begin{corollary}
  	\label{cor-double}
  	For any simple graph $G$, let $(G', \tau)$ be obtained from $G$ by replacing each edge of $G$ by a positive edge and a negative edge. Then $\chi_c(G', \tau) = 2 \chi_c(G)$.
 \end{corollary}

For a graph $G$ and an arbitrary signature $\sigma$, with the definition of $(G', \tau)$ given in the previous corollary, we have $(G, \sigma) \subset (G',\tau)$, thus
 
\begin{corollary}\label{prop-double}
 For every graph $G$, $\chi_c^s(G) \le 2 \chi_c(G)$.
\end{corollary}

 As adding or deleting negative loops does not affect the circular chromatic number, the signed  graph $(G, \sigma)$ obtained from $K_{p;q}$ by replacing each edge with a pair of positive and negative edges has circular chromatic number $\frac{2p}{q}$. So Corollary~\ref{prop-double} is tight. However, this signed graph has girth 2, i.e., has parallel edges.  
 The following result shows that the bound in Corollary~\ref{prop-double} is also tight for graphs of large girth. 
 
 \begin{theorem}
 	\label{thm-largegirth}
 	For any integers $k, g \ge 2$, for any $\epsilon > 0$, there is a graph $G$ of girth at least $g$ satisfying that $\chi(G)=k$ and $\chi_c^s(G) > 2k- \epsilon$. 
 \end{theorem} 
 
 The proof of Theorem \ref{thm-largegirth} uses the concept of augmented tree introduced in \cite{ABKWZ}. A \emph{complete $k$-ary tree} is a rooted tree in which each non-leaf vertex has $k$ children and all the leaves are of the same level (the level of a vertex $v$ is the distance from $v$ to the root).  For a leaf $v$ of $T$, let $P_v$ be the unique path in $T$ from the root to $v$. Vertices in $P_v -\{v\}$ are {\em ancestors} of $v$. 
An \emph{$q$-augmented $k$-ary tree} is obtained from a complete $k$-ary tree by adding, for each leaf $v$, $q$ edges connecting $v$ to $q$ of its ancestors. These $q$ edges are called the {\em augmenting edges} from $v$.
 For positive integers $k,q,g$, a $(k,q,g)$-graph is a $q$-augmented $k$-ary tree which is bipartite and has girth at least $g$.
 The following result was proved in \cite{ABKWZ}.
 
 \begin{lemma}
 	\label{lem-kqggraph} For any positive integers $k,q,g \ge 2$, there exists a $(k,q,g)$-graph.
 \end{lemma}
 
 Assume $T$ is a complete $k$-ary tree. A \emph{standard labeling} of the edges of $T$ is a labeling $\phi$ of the edges of $T$ such that for each non-leaf vertex $v$, for each $i \in \{1,2,\ldots, k\}$, there is one edge from $v$ to one of its child   labeled by $i$. Given a $k$-coloring $f: V(T) \to [k]$ of the vertices of $T$ (which does not need to be proper), the $f$-path $P_f = (v_1, v_2, \ldots, v_m)$ of $T$ is the path from the root vertex $v_1$ to a leaf $v_m$ of $T$ so that for each $i=1,2,\ldots, m-1$, $f(v_i) = \phi(v_iv_{i+1})$. 
 
 \medskip
 \noindent
 {\bf Proof of Theorem \ref{thm-largegirth}}
  
  Assume $k, g \ge 2$ are   integers. 
  We shall prove that for any integer $p$, there is a   graph $G$ for which the followings hold:
  \begin{enumerate}
  	\item $G$ has  girth at least $g$ and chromatic number at most $k$.
  	\item There is a signature $\sigma$ of $G$ such that $(G, \sigma)$ is not $(2kp, (p+1))$-colorable.
  \end{enumerate}
  
    Let $H$ be a $(2kp, k, 2kg )$-graph with underline tree $T$.  Let $\phi$ be a standard $2kp$-labeling of the edges of  $T$. 
  For $v \in V(T)$, denote by $\ell(v)$ the level of $v$, i.e., the distance from $v$ to the root vertex in $T$. Let $\theta(v) = \ell(v) \pmod{k}$. 
  
  For each leaf $v$ of $T$, 
  let $u_{v,1}, u_{v,2}, \ldots, u_{v, k}$ be the vertices on $P_v$ that are connected to $v$ by augmenting edges. Let $u'_{v,i} \in P_v$ be the closest descendant of $u_{v,i}$ with $\theta(u'_{v,i}) = i$ and let $e_{v,i}$ be the edge connecting $u'_{v,i}$ to its child on $P_v$. 
  
  Let $s_{v,i} = \phi(e_{v,i})$ and let 
  $$A_{v,i} = \{s_{v,i}, s_{v,i}+1, \ldots, s_{v,i}+p \}, \ B_{v,i} = \{a+kp: a \in A_{v,i}\}, 
  \ C_{v,i} = A_{v,i} \cup B_{v,i}.$$
  The addition above are carried out modulo $2kp$.
  
 As $|C_{v,i}| =2(p+1)$ and  $\cup_{i=1}^k C_{v,i} \subseteq [2kp]$, there exist distinct indices $i,j$ such that $C_{v,i} \cap C_{v,j} \ne \emptyset$. 
  
  Note that $B_{v,i}$ is a $kp$-shift of $A_{v,i}$. So if $A_{v,i} \cap A_{v,j} \ne \emptyset$, then $B_{v,i} \cap B_{v,j} \ne \emptyset$. In this case, 
  $$   d_{\pod{2kp}}(  \phi(e_{v,i}), \phi(e_{v,j})) \le p.$$
  Otherwise $A_{v,i} \cap B_{v,j} \ne \emptyset$ (and hence  $B_{v,i} \cap A_{v,j} \ne \emptyset$) and  $$ d_{\pod{2kp}}(  \phi(e_{v,i}), \overbar{\phi(e_{v,j}) })\le p.$$

  Let $L$ be the set of leaves of $T$. For each $v \in L$, we define one edge $e_v$ on $V(T)$ as follows:
  \begin{itemize}
  	\item If $d_{\pod{2kp}}(  \phi(e_{v,i}), \phi(e_{v,j})) \le p$, then let $e_v$ be a positive edge connecting $u'_{v,i}$ and $u'_{v,j}$.
  	\item If $d_{\pod{2kp}}(  \phi(e_{v,i}), \overbar{\phi(e_{v,j})}) \le p$, then let $e_v$ be a negative edge connecting $u'_{v,i}$ and $u'_{v,j}$.
  \end{itemize}
  Let $(G, \sigma)$ be the signed graph with vertex set $V(T)$ and with edge set $\{e_v: v \in L\}$, where the signs of the edges are defined as above.
  We shall show that $(G, \sigma)$ has the desired properties.
  
  First observe that $\theta$ is a proper $k$-coloring of $G$. So 
  $G$   has chromatic number at most $k$.
  
  Next we show that $G$ has girth at least $g$.
  For each edge $e_v = u'_{v,i}u'_{v,j}$ of $G$, let $B_v$ be the path which is the union of the subpath of $P_v$ from   $u'_{v,i}$ to $u_{v,i}$ and the path $u_{v,i}vu_{v,j}$ and the subpath of $P_v$ from $u_{v,j}$ to $u'_{v,j}$. Then $B_v$ has length at most $2k$. If $C$ is a cycle in $G$, then replace each edge $e_v$ of $C$ by the path $B_v$, we obtain a cycle in $H$. As $H$ has girth at least $2kg$, we conclude that $C$ has length at least $g$ and hence $G$ has girth at least $g$. 
  
  Finally, we show that $(G, \sigma)$ is not circular $(2kp, p+1)$-colorable.  Assume $f$ is a 
  $(2kp, p+1)$-colorable of $(G, \sigma)$. As $f$ is a $2kp$-coloring of the vertices of $T$, there is a unique $f$-path $P_v$. Assume $e_v = u'_{v,i}u'_{v,j}$. It follows from the definition of $f$-path that $f(u'_{v,i}) = \phi(e_{v,i})$ and $f(u'_{v,j}) = \phi(e_{v,j})$. It follows from the definition of $e_v$ that if $e_v$ is a positive edge, then 
  $d_{\pod{2kp}}(\phi(e_{v,i}), \phi(e_{v,j})) \le p$. If $e_v$ is a negative edge, then 
  $d_{\pod{2kp}}(\phi(e_{v,i}), \overbar{\phi(e_{v,j})}) \le p$. This is in contrary to the assumption that $f$ is   a circular $(2kp, p+1)$-coloring of $(G, \sigma)$.
  \qed
  
 Remark: The graph constructed above is shown to have chromatic number at most $k$. However, since $ \frac{2kp}{p+1} < \chi_c(G, \sigma) \le 2 \chi(G)$, we conclude that $\chi(G) = k$  when $p+1 \ge 2k$.   It is not known whether there is a finite $k$-chromatic graph of girth at least $g$ and with $\chi_c^s(G) =2k$. Also it is unknown whether for every rational $\frac{p}{q}$ and integer $g$ and any $\epsilon > 0$, there is a graph $G$ with $\chi_c(G) \le  \frac{p}{q}$ and $\chi_c^s(G) > \frac{2p}{q} - \epsilon$.

 A graph $G$ is called {\em $k$-critical} if $\chi(G)=k$ and for any proper subgraph $H$ of $G$, $\chi(H)=k-1$. 
 The following result about circular chromatic number of critical graphs of large girth was proved in \cite{Z01}.  
 
 \begin{theorem}
 	\label{thm-critical-circ}
 	For any integer $k \ge 3$ and $\epsilon > 0$, there is an integer $g$ such that any $k$-crtical graph of girth at least $g$ has circular chromatic number at most $k-1+\epsilon$.
 \end{theorem}
 
 As a consequence of Theorem \ref{thm-critical-circ} and Corollary \ref{prop-double}, we know that for any integer $k \ge 3$ and $\epsilon > 0$, there is an integer $g$ such that any $k$-critical graph $G$ of girth at least $g$ has signed circular chromatic number at most $2k-2+\epsilon$. However, this bound is not tight. The following proposition follows from Proposition \ref{prop-0free}.
 	
 	\begin{proposition}
 		\label{prop-critical}
 		If $G$ is a $k$-critical graph, then $\chi^s_c(G) \le 2k-2$.
 	\end{proposition}
 	\begin{proof}
  Let $\sigma$ be a signature on $G$. If $(G, \sigma) = (G, +)$, then $\chi_c(G, \sigma) \le \chi(G, \sigma)=\chi(G)=k$.
  If $\sigma(e)=-$ for some edge $e$, then the subgraph of $G$ induced by positive edges has chromatic number at most $k-1$. Hence $\chi_c(G, \sigma) \le 2(k-1)$.
  \end{proof}

  \section{Signed indicator}
  \label{sec:signedindicator}
  
In the study of coloring and homomorphism of graphs, using gadgets to construct new graphs from old ones is a fruitful tool. In this section, we explore the same idea for signed graph coloring.

  \begin{definition}
  	A {\em signed indicator} $\mathcal{I}$ is a triple $\mathcal{I} =(\Gamma, u,v)$ such that $\Gamma$ is a signed graph and $u, v$ are two distinct vertices of $\Gamma$. 
  \end{definition}
 
 \begin{definition}
 	 Assume $\Omega$ is a signed graph,   $\mathcal{I}=(\Gamma, u,v)$ is a signed indicator and $e=xy$ is an  (either positive or negative) edge of $\Omega$. By {\em replacing $e$ with a copy of $\mathcal{I}$}, we mean the following operation: Take the disjoint union of $\Omega$ and $\mathcal{I}$, delete the edge $e$ from $\Omega$, identify $x$ with $u$ and identify $y$ with $v$. 
 \end{definition}

  There is a subtle issue in the above definition. An edge $e=xy$ is an unordered pair. So we can write it as $e=yx$ as well. However, by identifying $y$ with $u$ and identifying $x$ with $v$, the resulting signed graph is different from the one as defined above. To avoid such confusion, it is safer to first orient the edges of $\Omega$ and then replace the directed edge $e$ with $\mathcal{I}$. However, for our usage in this paper, the difference does not affect our discussion, so we just say  replace the edge $e$ with $\mathcal{I}$. 
  
  \begin{definition}
         For a graph $G$ and a signed indicator $\mathcal{I}$, we denote by $G(\mathcal{I})$ the signed graph obtained from $G$ by replacing each edge with a copy of $\mathcal{I}$.
  	 Assume $\Omega$ is a signed graph, and $\mathcal{I}_+$ and $\mathcal{I}_-$ are two signed indicators. We denote by $\Omega(\mathcal{I}_+, \mathcal{I}_-)$ the signed graph obtained from $\Omega$ by replacing each positive edge   with a copy $\mathcal{I}_+$ and replacing each negative edge with a copy of $\mathcal {I}_-$.   
  \end{definition}

\begin{definition}
  	  Assume $\mathcal{I}=(\Gamma,u,v)$ is a signed indicator and $r \ge 2$ is a real number. 
  	  For $a, b \in [0,r)$, we say the color pair $(a,b)$ is {\em feasible for $\mathcal{I}$} (with respect to $r$) if    there is a circular $r$-coloring $\phi$ of $\Gamma$ such that $\phi(u)=a$ and $\phi(v)=b$.
\end{definition}
  
  Note that if $(a, b)$ is feasible for $\mathcal{I}$, then for any $t \in [0,r)$, $(a+t, b+t)$ and $(-a, -b)$ are also feasible for $\mathcal{I}$. Here the calculation is modulo $r$. Thus if we know feasible pairs of the form $(0,b)$ for $b \in [0, \frac{r}{2}]$, then we know all the feasible pairs.

\begin{definition}
Assume $\mathcal{I}=(\Gamma,u,v)$ is a signed indicator and $r \ge 2$ is a real number. Let $$Z(\mathcal{I}, r) = \{b \in [0, \dfrac{r}{2}] : (0,b) \text{ is feasible for $\mathcal{I}$ with respect to $r$ } \}.$$
\end{definition}

  Observe that for $\mathcal{I}=(\Gamma, u,v)$, $Z(\mathcal{I},r) \ne \emptyset$ if and only if $\chi_c(\Gamma) \le r$. One useful interpretation of $Z(\mathcal{I},r)$ is that this is the set of possible distances (in $C^r$) between the two colors assigned to $u$ and $v$ in a circular $r$-coloring of $\Gamma$. 
  
Let the sign of a path $P$ in $(G, \sigma)$ be the product of the signs of the edges of $P$.
  
  \begin{example}
  	\label{ex-gadget}
  	If $\Gamma$ is a positive $2$-path connecting $u$ and $v$, and $\mathcal{I}=(\Gamma,u,v)$, then for any $\epsilon$, $0<\epsilon <1 $, and $r=4-2\epsilon$, $$Z(\mathcal{I}, r) = [0, 2-2\epsilon]=[0, \dfrac{r}{2} - \epsilon].$$
  	If $\Gamma'$ is a negative $2$-path connecting $u$ and $v$, and $\mathcal{I}'=(\Gamma',u,v)$,  then for any $\epsilon$, $0<\epsilon <1 $, and $r=4-2\epsilon$, $$Z(\mathcal{I}', r) = [\epsilon, \dfrac{r}{2}].$$
  	If $\Gamma''$ consists of a negative $2$-path and a positive 2-path connecting $u$ and $v$, and $\mathcal{I}''=(\Gamma'',u,v)$,  then for any $\epsilon$, $0<\epsilon <1 $, and $r=4-2\epsilon$, $$Z(\mathcal{I}'', r) = [\epsilon, \dfrac{r}{2} - \epsilon].$$
  \end{example}
    
  \begin{lemma}
  	\label{lem-indicator}
  	Assume $\mathcal{I}=(\Gamma,u,v)$ is a signed indicator, $r \ge 2$ is a real number and 
  	  $$Z(\mathcal{I}, r) =[t,  \dfrac{r}{2} - t]$$
  	for some $0< t < \frac{r}{4}$. Then for any graph $G$, 
  	$$\chi_c(G) = \frac{\chi_c(G(\mathcal{I}))}{2t}.$$ 
  \end{lemma}
  
  \begin{proof}
  	Let $r' = \frac{r}{2t}$. If $\chi_c(G) \le r'$ and $f$ is a circular $r'$-coloring of $G$, then $g: V(G) \to [0,r)$ defined as $g(x) = tf(x)$ satisfies the condition that for any edge $e=xy$ of $G$,  $$t \le d_{\pod{r}}(g(x),g(y)) \le \frac{r}{2} - t.$$
  	So $d_{\pod{r}}(g(x),g(y)) \in Z(\mathcal{I},r)$, and the mapping $g$ can be extended to a circular $r$-coloring of the copy of $\Gamma$ that was used to replace $e$. 
  	So $g$ can be extended to a circular $r$-coloring of $G(\mathcal{I})$.
  	
  	Conversely, assume $\chi_c(G(\mathcal{I})) \le r$. Let $g$ be a circular $r$-coloring of $G(\mathcal{I})$.  By vertex switching, we may assume that $g(x) \in [0, \frac{r}{2})$ for every vertex $x$ of $G(\mathcal{I})$. 
  	Then for any edge $e=xy$ of $G$, $d_{\pod{r}}(g(x), g(y)) \in Z(\Gamma, r)$, i.e., 
  	$t \le d_{\pod{r}}(g(x), g(y)) \le  \frac{r}{2} -t$.
  	Let $f: V(G) \to [0,r')$  be defined as $f(x) = \frac{1}{t}g(x)$. Then for any edge $e=xy$ of $G$, $1 \le |f(x)-f(y)| \le r'-1$. Hence $f$ is a circular $r'$-coloring of $G$. 
  \end{proof}

A similar proof implies the following:

\begin{lemma}
  \label{lem-signedindicator}
  Assume $\mathcal{I}_+$ and $\mathcal{I}_-$ are indicators, $r \ge 2$ is a real number and 
  	$$Z(\mathcal{I}_+, r) =[t,  \dfrac{r}{2}], Z(\mathcal{I}_-, r) =[0,   \dfrac{r}{2}-t]$$
  	for some $0< t < \frac{r}{2}$.
  	Then for any signed graph $\Omega$, 
  	$$\chi_c(\Omega) = \frac{\chi_c(\Omega(\mathcal{I}_+, \mathcal{I}_-))}{t}.$$ 
\end{lemma}

  \begin{corollary}
  	Let $\mathcal{I} = (\Gamma, u,v)$ be the indicator, where $\Gamma$ consists of a positive $2$-path and a negative 2-path connecting $u$ and $v$. Then for any graph $G$, $$\chi_c(G(\mathcal{I})) = 4-\frac{4}{\chi_c(G)+1}.$$
  \end{corollary}
  \begin{proof}
  	Let $\epsilon = \frac{2}{r'+1}$ and $r=4-2\epsilon$. By Example \ref{ex-gadget}, $Z(\mathcal{I}, r) = [\epsilon, \frac{r}{2} - \epsilon]$. Note that $r' = \frac{r}{2\epsilon}$.
  	The conclusion follows from Lemma \ref{lem-indicator}. 
 \end{proof}
   
We note that $G(\mathcal{I})$ here is the same as $S(G)$ defined in \cite{NRS15}. In \cite{NRS15}, it is shown that by using $S(G)$ construction and the graph homomorphism, the chromatic number of graphs are captured by switching homomorphisms of signed bipartite graphs. This corollary shows that furthermore $\chi_c(S(G))$ also determines $\chi_c(G)$.

   \begin{definition}
   		Let  $\Gamma_1$ be  a   positive 2-path connecting $u_1$ and $v_1$. For $i \ge 2$,  
   		\begin{itemize}
   			\item if $i$ is even, then let $\Gamma_{i}$ be obtained from $\Gamma_{i-1}$ by 
   			\begin{itemize}
   				\item adding 
   				two vertices $u_{i}, v_{i}$,
   				\item connecting $u_{i}$ to $u_{i-1}$ by a positive edge, $u_{i}$ to $v_{i-1}$ by a negative edge,
   				\item connecting  $v_{i}$ to $u_{i-1}$ by a negative edge, $v_{i}$ to $v_{i-1}$ by a positive edge;
   			\end{itemize} 
   			\item if $i$ is odd, then let $\Gamma_{i}$ be obtained from $\Gamma_{i-1}$ by 
   				\begin{itemize}
   					\item adding 
   					two vertices $u_{i}, v_{i}$,
   					\item connecting each of $u_{i}$ and $v_{i}$ to each of $u_{i-1}$ and  $v_{i-1}$ by a positive edge.
   				\end{itemize} 
   		\end{itemize}
   	\end{definition}
	
For example, $\Gamma_{3}, \Gamma_{4}$ and $\Gamma_{5}$ are illustrated in Figure~\ref{fig:Gamma3} \ref{fig:Gamma4} \ref{fig:Gamma5} respectively.

\begin{figure}[htbp]
	\centering
	\begin{minipage}{.3\textwidth}
		\centering
		\begin{tikzpicture}
		[scale=.3]
\draw[rotate=0] (0, 0) node[line width=0mm, circle, fill, inner sep=1.5pt](x0){};
\draw[rotate=0] (-2, 0) node[line width=0mm, circle, fill, inner sep=1.5pt, label=left: $u_1$](u1){};
\draw[rotate=0] (2, 0) node[line width=0mm, circle, fill, inner sep=1.5pt, label=right: $v_1$](v1){};
\draw[rotate=0] (0, 2) node[line width=0mm, circle, fill, inner sep=1.5pt, label=above: $u_2$](u2){};
\draw[rotate=0] (0, -2) node[line width=0mm, circle, fill, inner sep=1.5pt, label=below: $v_2$](v2){};
\draw[rotate=0] (-5, 0) node[line width=0mm, circle, fill, inner sep=1.5pt, label=left: $u_3$](u3){};
\draw[rotate=0] (5, 0) node[line width=0mm, circle, fill, inner sep=1.5pt, label=right: $v_3$](v3){};

\draw  [dashed, line width=0.6mm, red] (u2) -- (v1);
\draw  [dashed, line width=0.6mm, red] (v2) -- (u1);
\draw  [line width=0.6mm, blue] (u2) -- (u1) -- (x0) -- (v1) -- (v2);
\draw  [line width=0.6mm, blue] (u2) -- (v3) -- (v2) -- (u3) -- (u2);
		
		\end{tikzpicture}
		\caption{$\Gamma_3$}
		\label{fig:Gamma3}     
	\end{minipage}
	\begin{minipage}{.3\textwidth}
		\centering
		\begin{tikzpicture}
		[scale=.3]
\draw[rotate=0] (0, 0) node[line width=0mm, circle, fill, inner sep=1.5pt](x0){};
\draw[rotate=0] (-2, 0) node[line width=0mm, circle, fill, inner sep=1.5pt, label=left: $u_1$](u1){};
\draw[rotate=0] (2, 0) node[line width=0mm, circle, fill, inner sep=1.5pt, label=right: $v_1$](v1){};
\draw[rotate=0] (0, 2) node[line width=0mm, circle, fill, inner sep=1.5pt, label=above: $u_2$](u2){};
\draw[rotate=0] (0, -2) node[line width=0mm, circle, fill, inner sep=1.5pt, label=below: $v_2$](v2){};
\draw[rotate=0] (-5, 0) node[line width=0mm, circle, fill, inner sep=1.5pt, label=left: $u_3$](u3){};
\draw[rotate=0] (5, 0) node[line width=0mm, circle, fill, inner sep=1.5pt, label=right: $v_3$](v3){};
\draw[rotate=0] (0, 5) node[line width=0mm, circle, fill, inner sep=1.5pt, label=above: $u_4$](u4){};
\draw[rotate=0] (0, -5) node[line width=0mm, circle, fill, inner sep=1.5pt, label=below: $v_4$](v4){};

\draw  [dashed, line width=0.6mm, red] (u2) -- (v1);
\draw  [dashed, line width=0.6mm, red] (v2) -- (u1);
\draw  [line width=0.6mm, blue] (u2) -- (u1) -- (x0) -- (v1) -- (v2);
\draw  [line width=0.6mm, blue] (u2) -- (v3) -- (v2) -- (u3) -- (u2);
\draw  [dashed, line width=0.6mm, red] (u4) -- (v3);
\draw  [dashed, line width=0.6mm, red] (v4) -- (u3);
\draw  [line width=0.6mm, blue] (u4) -- (u3);
\draw  [line width=0.6mm, blue] (v4) -- (v3);
	\end{tikzpicture}
	\caption{$\Gamma_4$}
	\label{fig:Gamma4}
	\end{minipage}
\begin{minipage}{.3\textwidth}
		\centering
		\begin{tikzpicture}
		[scale=.3]
\draw (0, 0) node[line width=0mm, circle, fill, inner sep=1.5pt](x0){};
\draw (-2, 0) node[line width=0mm, circle, fill, inner sep=1.5pt, label=left: $u_1$](u1){};
\draw[rotate=0] (2, 0) node[line width=0mm, circle, fill, inner sep=1.5pt, label=right: $v_1$](v1){};
\draw[rotate=0] (0, 2) node[line width=0mm, circle, fill, inner sep=1.5pt, label=above: $u_2$](u2){};
\draw[rotate=0] (0, -2) node[line width=0mm, circle, fill, inner sep=1.5pt, label=below: $v_2$](v2){};
\draw[rotate=0] (-5, 0) node[line width=0mm, circle, fill, inner sep=1.5pt, label=left: $u_3$](u3){};
\draw[rotate=0] (5, 0) node[line width=0mm, circle, fill, inner sep=1.5pt, label=right: $v_3$](v3){};
\draw[rotate=0] (0, 5) node[line width=0mm, circle, fill, inner sep=1.5pt, label=above: $u_4$](u4){};
\draw[rotate=0] (0, -5) node[line width=0mm, circle, fill, inner sep=1.5pt, label=below: $v_4$](v4){};
\draw (-8, 0) node[line width=0mm, circle, fill, inner sep=1.5pt, label=left: $u_5$](u5){};
\draw (8,0) node[line width=0mm, circle, fill, inner sep=1.5pt, label=right: $v_5$](v5){};

\draw  [dashed, line width=0.6mm, red] (u2) -- (v1);
\draw  [dashed, line width=0.6mm, red] (v2) -- (u1);
\draw  [line width=0.6mm, blue] (u2) -- (u1) -- (x0) -- (v1) -- (v2);
\draw  [line width=0.6mm, blue] (u2) -- (v3) -- (v2) -- (u3) -- (u2);
\draw  [dashed, line width=0.6mm, red] (u4) -- (v3);
\draw  [dashed, line width=0.6mm, red] (v4) -- (u3);
\draw  [line width=0.6mm, blue] (u4) -- (u3);
\draw  [line width=0.6mm, blue] (v4) -- (v3);
\draw  [line width=0.6mm, blue] (u5) -- (u4) -- (v5) -- (v4) -- (u5);
	\end{tikzpicture}
	\caption{$\Gamma_5$}
	\label{fig:Gamma5}
	\end{minipage}
\end{figure}

   	\begin{lemma}
   		\label{lem-gammai}
   		Assume $i \ge 1$, $0< \epsilon < \frac{1}{i}$ and $r=4-2\epsilon$. The followings hold:
   		\begin{itemize} 
   			\item If $i$ is odd, then  $$Z(\mathcal{I}_{i}, r) =[0,   \frac{r}{2}-i \epsilon].$$ 
   	                 \item If $i$ is even, then $$Z(\mathcal{I}_{i},r)   = [i \epsilon, \frac{r}{2}].$$
   \end{itemize}
   	\end{lemma}
   	
   		\begin{proof} We prove the lemma by induction on $i$. For $i=1$, this is trivial and observed in Example \ref{ex-gadget}.
   			
   			Assume $i \ge 2$ and the lemma holds for $i' < i$. 
   			
   			\noindent
   			{\bf Case 1} $i$ is even.

   			 Assume $\phi$ is a circular $r$-coloring of $\Gamma_i$   with $\phi(u_{i-1})=0$.  
   			   As $Z(\mathcal{I}_{i-1},r) = [0, \frac{r}{2}-(i-1)\epsilon]$, we may assume that $\phi(v_{i-1}) \in [0, \frac{r}{2}-(i-1)  \epsilon]$.
   			 
   			 The possible colors for $u_i$  are $[1, \frac{r}{2}+1-i \epsilon]$, and the possible colors for $v_i$ are 
   			 $[3-\epsilon, r) \cup [0, 1-i \epsilon]$. 
   			 So the possible distances between $\phi(u_i)$ and $\phi(v_i)$ are $[ i\epsilon, \frac{r}{2}]$, i.e., $$Z(\mathcal{I}_i, r) = [  i\epsilon, \frac{r}{2}].$$
   			 
   				\noindent
   				{\bf Case 2} $i$ is odd.
   				 
   			   Assume $\phi$ is a circular $r$-coloring of $\Gamma_i$ with $\phi(u_{i-1})=0$.  
   			   As $Z(\mathcal{I}_{i-1},r) = [(i-1)\epsilon, \frac{r}{2}]$, we may assume that $\phi(v_{i-1}) \in [(i-1)\epsilon, \frac{r}{2}]$.
   			   
   			   The possible colors for $u_i$ and $v_i$ are $[1+  (i-1) \epsilon, 3-2\epsilon]$. So the possible distances between $\phi(u_i)$ and $\phi(v_i)$ are $[0, 2-(i+1)\epsilon] =[0, \frac{r}{2} - i\epsilon]$, i.e., $$Z(\mathcal{I}_{i},r) = [0, \frac{r}{2}- i \epsilon].$$ 
\end{proof}

   	\begin{corollary}
   			\label{cor-bsp}
   			For any $\epsilon > 0$, there is a signed bipartite planar simple graph $\Gamma$ with 
   			$\chi_c(\Gamma) > 4-2\epsilon$.
   	\end{corollary}	 
   	\begin{proof}
   		Let $\frac{1}{2\epsilon}< i < \frac{1}{\epsilon}$. Let $\Gamma'_i$ be obtained from the disjoint union of $\Gamma_{2i-1}$ and $\Gamma_{2i}$ by identifying   $u_{2i-1}$ in $\Gamma_{2i-1}$ and $u_{2i}$ in  $\Gamma_{2i}$ into a single vertex $u'_i$, and identifying $v_{2i-1}$ in $\Gamma_{2i-1}$ and $v_{2i}$ in  $\Gamma_{2i}$ into a single vertex $v'_i$.
   		It follows from the construction that $\Gamma'_i$ is a signed bipartite planar simple graph.
   		
   		Let $\mathcal{I}'_i = (\Gamma'_i, u'_i, v'_i)$. Then for $r = 4 - 2\epsilon$, $\frac{r}{2}-(2i-1)\epsilon = 2-2i\epsilon < 2i\epsilon$. Hence $$Z(\mathcal{I}'_i,r) = Z(\mathcal{I}_{2i-1},r) \cap Z(\mathcal{I}_{2i},r) = [0, \dfrac{r}{2}-(2i-1)\epsilon] \cap [2i\epsilon, \dfrac{r}{2}] = \emptyset.$$
   		So $\Gamma'_i$ is not circular $r$-colorable.
   	\end{proof} 
   		
   \begin{corollary}
   	\label{cor-i2i}
   	If $i=2k$, then for any graph $G$, 
   	$$\chi_c(G(\mathcal{I}_i)) = \frac{4k\chi_c(G)}{k\chi_c(G)+1}=4-\frac{4}{k\chi_c(G)+1}.$$
   \end{corollary}
  \begin{proof}
  	By Lemma~\ref{lem-gammai}, $Z(\mathcal{I}_i) = [i\epsilon, \frac{r}{2}]$, where $i\epsilon = k(4-r)$. The conclusion follows from Lemma \ref{lem-indicator}.
  \end{proof}

\section{Circular chromatic number of signed  graph classes}\label{sec:GraphClasses}

We have shown that $\chi_c^s(G) \le 2 \chi_c(G)$ and this bound is tight even for graphs $G$ of large girth. However, when restricted to some natural families of graphs, the upper bound can be improved.

Given a class $\mathcal{C}$ of signed graphs we define $\chi_c(\mathcal{C})=\sup\{\chi_c(G,\sigma): (G,\sigma) \in \mathcal{C} \}$. In light of Corollary~\ref{cor-double} and the fact that negative loops do not affect the circular chromatic number, we shall restrict to signed graphs with no negative digons and no loops, i.e., the underlying graphs are simple graphs.  

 We denote by 
 \begin{itemize}
 	\item $\mathcal{SD}_d$ the class of signed $d$-degenerate simple graphs,
 	\item  $\mathcal{SSP}$ the class of signed series parallel simple graphs,
 	\item $\mathcal{O}$   the class of signed outer planar  simple graphs,
 	\item $\mathcal{SBP}$ the class of signed bipartite planar simple graphs,
 	\item  $\mathcal{SP}$ the class of signed planar simple graphs.
 \end{itemize}

 \begin{proposition}\label{prop-d-degenerated}
 For any positive integer $d$, $\chi_c(\mathcal{SD}_d)= 2\floor{\dfrac{d}{2}}+2$.
 \end{proposition}
 
 \begin{proof}
  	First we show that every $(G, \sigma)\in \mathcal{SD}_d$ admits a circular $(2\floor{\frac{d}{2}}+2)$-coloring. Equivalently, $(G, \sigma)$ admits an edge-sign preserving homomorphism to $K^s_{2\floor{\frac{d}{2}}+2}$ whose vertices are labelled $0,1, \ldots, 2\floor{\frac{d}{2}}+1$ in a cyclic order. Recall that in $K^s_{2\floor{\frac{d}{2}}+2}$ between any pair of vertices $x_i, x_j$ there are both positive and negative edges, unless $i=j$ or $i=j+\floor{\frac{d}{2}}+1$. When $i=j$, there is a negative loop but no positive loop; when $i=j+\floor{\frac{d}{2}}+1$, $x_ix_j$ is a positive edge but not a negative edge. Thus, given a vertex $u$ of $(G, \sigma)$ and a partial mapping $\phi$ of $(G, \sigma)$ to $K^s_{2\floor{\frac{d}{2}}+2}$, if at most $d$ neighbors of $u$ are already colored, then $\phi$ can be extended to $u$. This now can be applied on the ordering of vertices of $G$ which is a witness of $G$ being $d$-degenerate.
	
	To prove that the upper bound is tight, we consider three cases. For $d=2$, the signed graphs built in Corollary~\ref{cor-bsp} are all $2$-degenerate and the claim of this corollary is that the limit of their circular chromatic number is $4$. For odd integer $d$, this bound is tight by considering the signed complete graphs $(K_{d+1}, +)$. For even integer $d\geq 4$, we now construct a $d$-degenerate graph $G$ together with a signature $\sigma$ such that $\chi_c(G, \sigma)=d+2$.

	Define a signed graph $\Omega_d$ as follows. Take $(K_d, +)$ whose vertices are labelled $x_1, x_2, \ldots, x_d$. For each pair $i,j \in [d]$ ($ i\neq j$), we add a vertex $y_{i,j}$ and join it to $x_i$, $x_j$ with negative edges, and to all the other $x_k$'s with positive edges. Since each $y_{i,j}$ is of degree $d$ and after removing all of them we are left with a $K_d$, we have $\Omega_d\in \mathcal{SD}_d$. We claim that $\chi_c(\Omega_d)=d+2$.

	Assume this is not true and   $\varphi$ is a circular $r$-coloring of $\Omega_d$ and $r < d+2$.   Without loss of generality, we may assume that $\varphi(x_1), \varphi(x_2), \ldots, \varphi(x_d)$ are cyclicly ordered on $C^r$ in a clockwise orientation. Furthermore, we may also assume that $\varphi(x_1), \varphi(x_2)$ has the maximum distance among all the pairs $\varphi(x_i), \varphi(x_{i+1})$ where the addition of the index is taken $\pmod{2d}$. As the distance between each consecutive pair $\varphi(x_i), \varphi(x_{i+1})$ is at least $1$, it follows that, except for $x_1,x_2$, $d_{\pod r}(\varphi(x_i), \varphi(x_{i+1})) < 2$. We will now show that there is no possible choice for $y_{1,1+\frac{d}{2}}$. A point between $\varphi(x_i)$ and $\varphi(x_{i+1})$ for $i\in \{2, 3, \ldots,\frac{d}{2}-1\}\cup \{ \frac{d}{2}+2,\ldots,  d-1\}$ is at distance less than $1$ from one of the two and cannot be the color of $y_{1,1+\frac{d}{2}}$ because $x_iy_{1,1+\frac{d}{2}}$, $x_{i+1}y_{1,1+\frac{d}{2}}$ are both positive edges.   If $\varphi(y_{1,1+\frac{d}{2}})\in [\varphi(x_1), \varphi(x_2)]$, then we show that $d_{\pod r}(\varphi(y_{1,1+\frac{d}{2}}), \varphi(x_{1+\frac{d}{2}}) \geq \frac{d}{2}$, which is a contradiction because $y_{1,1+\frac{d}{2}}x_{1+\frac{d}{2}}$ is a negative edge.   To see this, we consider clockwise and anti-clockwise distances of $\varphi(y_{1,1+\frac{d}{2}})$ and $\varphi(x_{1+\frac{d}{2}})$. On the anti-clockwise direction, $(\varphi(y_{1,1+\frac{d}{2}}), \varphi(x_{1+\frac{d}{2}}))$ contains $\frac{d}{2}$ intervals of the form $(x_i, x_{i+1})$, each of which is of length at least $1$. On the clockwise direction, first of all, $y_{1,1+\frac{d}{2}}x_2$ is a positive edge which means $d_{\pod r}(\varphi(y_{1,1+\frac{d}{2}}), \varphi(x_2)) \geq 1$, and, furthermore, $(\varphi(y_{1,1+\frac{d}{2}}), \varphi(x_{1+\frac{d}{2}}))$ contains $\frac{d}{2}-1$ intervals of form $(x_i, x_{i+1})$ (for $i\in \{2, 3, \ldots, \frac{d}{2}\}$). If $\varphi(y_{1,1+\frac{d}{2}})\in [\varphi(x_d), \varphi(x_1)]$, then the same argument shows that $d_{\pod r}(\varphi(y_{1,1+\frac{d}{2}}), \varphi(x_{1+\frac{d}{2}})) \geq \frac{d}{2}$. If $\varphi(y_{1,1+\frac{d}{2}})\in [\varphi(x_{\frac d2}), \varphi(x_{\frac d2 +1})]$ or $\varphi(y_{1,1+\frac{d}{2}})\in [\varphi(x_{\frac d2 +1}), \varphi(x_{\frac d2 +2})]$, then  $d_{\pod r}(\varphi(y_{1,1+\frac{d}{2}}), \varphi(x_1))\geq \frac{d}{2}$, which is a contradiction as  $y_{1,1+\frac{d}{2}}x_1$ is a negative edge.
 \end{proof}

It follows from Proposition \ref{prop-d-degenerated} that $\chi_c(G, \sigma)\leq 2\floor{\dfrac{\Delta(G)}{2}}+2$. 
  
It was proved in \cite{NRS15}  that every simple signed $K_4$-minor-free graph $(G, \sigma)$ admits a switching homomorphism to the signed Paley graph $SPal_5$, depicted in Figure \ref{fig:Paley}. It is easy to check that $SPal_5$ is a signed subgraph of $K^s_{10:3}$. Hence we have $$\chi_c(\mathcal{SO}) \le \chi_c(\mathcal{SSP}) \le   \dfrac{10}{3}.$$ 
  
\begin{figure}[htbp]
	\centering
	\begin{minipage}{.4\textwidth}
		\centering
		\begin{tikzpicture}
		[scale=.4]
\foreach \i in {1}
		{
			\draw[rotate=0] (0, 5) node[line width=0mm, circle, fill, inner sep=1.5pt, label=above: $\i$ ](\i){};
		}
		
		\foreach \i in {3}
		{
			\draw[rotate=72*4] (0, 5) node[line width=0mm, circle, fill, inner sep=1.5pt, label=right: $\i$ ](\i){};
		}
		
		\foreach \i in {5}
		{
			\draw[rotate=72*3] (0, 5) node[line width=0mm, circle, fill, inner sep=1.5pt, label={below right}: $\i$ ](\i){};
		}
		
		\foreach \i in {7}
		{
			\draw[rotate=72*2] (0, 5) node[line width=0mm, circle, fill, inner sep=1.5pt, label={below left}: $\i$ ](\i){};
		}
		
		\foreach \i in {9}
		{
			\draw[rotate=72] (0, 5) node[line width=0mm, circle, fill, inner sep=1.5pt, label=left: $\i$ ](\i){};
		}
		
		\foreach \i/\j in {1/5, 3/7, 5/9, 1/7, 3/9}
		{
			\draw  [line width=0.6mm, blue] (\i) -- (\j);
		}
		
		\foreach \i/\j in {1/3, 3/5, 5/7, 7/9, 1/9}
		{
			\draw  [dashed, line width=0.6mm, red] (\i) -- (\j);
		}
		\end{tikzpicture}
		\caption{The signed Paley graph}
		\label{fig:Paley}     
	\end{minipage}
	\begin{minipage}{.4\textwidth}
		\centering
		\begin{tikzpicture}
		[scale=.5]
	\foreach \i/\j in {1/y, 2/x, 3/z}
	{
		\draw[rotate=120*(\i-1)] (0,2.2) node[line width=0mm, circle, fill, inner sep=1.8pt, label=above: $\j$ ](\j){};
	}
	\foreach \i/\j in {1/b, 3/a, 5/c}
	{
		\draw[rotate=60*(\i)] (0,4.4) node[line width=0mm, circle, fill, inner sep=1.8pt, label=above: $\j$ ](\j){};
	}
	
	\foreach \i/\j in {y/b,x/a,z/c}
	{
		\draw  [line width=0.6mm, blue] (\i) -- (\j);
	}
	
	\foreach \i/\j in {y/x,y/z,y/c,x/b,x/z,z/a}
	{
		\draw  [dashed, line width=0.6mm, red] (\i) -- (\j);
	}
	
	\end{tikzpicture}
	\caption{$(F, \sigma)$}
	\label{fig:F}
	\end{minipage}
\end{figure}

We shall prove the following result.

\begin{theorem}
	\label{thm-OSP}   $\chi_c( \mathcal{SSP})=\chi_c( \mathcal{SO})=\dfrac{10}{3}$.
\end{theorem}
 \begin{proof}
	It suffices to show that $\chi_c(F, \sigma)= \frac{10}{3}$  for the signed outer planar simple graph $(F, \sigma)$ of Figure \ref{fig:F}. 
	
	Since $(F, \sigma)$ contains a positive triangle as a subgraph, its circular chromatic number is at least $3$. By the formula of the tight cycle the only possible values are $3$ and $\frac{10}{3}$. It remains to show that this graph does not admit a circular 3-coloring, that is to say, $(F, \sigma)$ does not admit a switching homomorphism to $\hat{K}^s_{6;2}$. Note that $\hat{K}^s_{6;2}$  is equivalent to a positive triangle, with each vertex incident to a negative loop. If $\phi$ is a switching homomorphism of $(F, \sigma)$ to $ \hat{K}^s_{6;2}$, then at least one negative edge of the negative triangle $xyz$ is mapped to a negative loop, because in $\hat{K}^s_{6;2}$ every negative closed walk contains a negative loop. Whichever edge of $xyz$ is mapped to a negative loop, its two end vertices are identified and the resulting signed graph has a negative cycle of length $2$. But $\hat{K}^s_{6;2}$ contains no negative even closed walk of length $2$, a contradiction. Hence $\chi_c(F, \sigma) = \frac{10}{3}$.    
\end{proof}
 
In Section~\ref{sec:signedindicator}, we have seen that $\chi_c(\mathcal{SBP})=4$. However, we do not know if there is a signed bipartite planar simple graph reaching the bound $4$. Further improvement based on the length of the shortest negative cycle is given in the forthcoming work \cite{NPW20}.

Next we consider the circular chromatic number of signed planar simple graphs. 
Since planar simple graphs are $5$-degenerate, by Proposition \ref{prop-d-degenerated}, we have $\chi_c(\mathcal{SP})\leq 6$. 
It was conjectured in \cite{MRS16} that every planar simple graph admits a $0$-free $4$-coloring. If the conjecture was true, it would have implied the best possible bound of 4 for the circular chromatic number of signed planar simple graphs. However, this conjecture was disproved in \cite{KN20} using a dual notion. A direct proof of a counterexample is given in \cite{NP20}. Extending this construction, we build a signed planar simple graph whose circular chromatic number is $4+\frac{2}{3}$.

\begin{theorem}\label{thm:PlanarGraph}
 $\chi_c( \mathcal{SP})\ge 4+\dfrac{2}{3}$. 
\end{theorem}
  We shall construct a signed planar simple graph $\Omega$ with $\chi_c(\Omega) =4+\frac{2}{3}$.  The construction is broken down into construction of certain gadgets.   
Similar to the gadget of \cite{KN20}, we start with a mini-gadget depicted in Figure~\ref{fig:minigadget} and state its circular coloring property in Lemma~\ref{lem:mingadget}.

\begin{figure}[htbp]
    \centering 
     \begin{minipage}{.4\textwidth}
    \centering
 \begin{tikzpicture}
	[scale=.5]	
   \foreach \i/\j in {1/a}
  {
    \draw[rotate=60*(\i)] (0,1.2) node[line width=0mm, circle, fill, inner sep=1.5pt, label=left: $\j$ ](\j){};
  }
   \foreach \i/\j in {3/b}
  {
    \draw[rotate=60*(\i)] (0,1.2) node[line width=0mm, circle, fill, inner sep=1.5pt, label=below: $\j$ ](\j){};
  }
   \foreach \i/\j in {5/c}
  {
    \draw[rotate=60*(\i)] (0,1.2) node[line width=0mm, circle, fill, inner sep=1.5pt, label=right: $\j$ ](\j){};
  }
  \foreach \i/\j in {1/x}
  {
    \draw[rotate=120*(\i)] (0,5) node[line width=0mm, circle, fill, inner sep=1.5pt, label=left: $\j$ ](\j){};
  }
  \foreach \i/\j in {2/y}
  {
    \draw[rotate=120*(\i)] (0,5) node[line width=0mm, circle, fill, inner sep=1.5pt, label=right: $\j$ ](\j){};
  }
\foreach \i/\j in {3/z}
  {
    \draw[rotate=120*(\i)] (0,5) node[line width=0mm, circle, fill, inner sep=1.5pt, label=above: $\j$ ](\j){};
  }

 \foreach \i/\j in {a/b, a/c, b/c, a/x, b/y, c/z}
  {
    \draw  [line width=0.6mm, blue] (\i) -- (\j);
  }
  
   \foreach \i/\j in {y/z,z/x,x/y,a/z,b/x,c/y}
  {
    \draw  [dashed, line width=0.6mm, red] (\i) -- (\j);
  }
  
        \end{tikzpicture}
        \caption{Mini-gadget $(T, \pi)$}
            \label{fig:minigadget}
 \end{minipage}
   \begin{minipage}{.4\textwidth}
    \centering
    	\begin{tikzpicture}
	[scale=.4]
	\begin{scope}[scale=.5] 
	\draw[rotate=0] (0,0)  node[line width=0mm, circle, fill, inner sep=1.5pt, label=right: $w$ ](w){};
	\draw[rotate=0] (0,9)  node[line width=0mm, circle, fill, inner sep=1.5pt, label=left: $x_1$ ](x1){};
	\draw[rotate=72] (0,9)  node[line width=0mm, circle, fill, inner sep=1.5pt, label=left: $x_2$ ](x2){};
	\draw[rotate=144] (0,9)  node[line width=0mm, circle, fill, inner sep=1.5pt, label={below left}: $x_3$ ](x3){};
	\draw[rotate=216] (0,9)  node[line width=0mm, circle, fill, inner sep=1.5pt, label={below right}: $x_4$ ](x4){};
	\draw[rotate=288] (0,9) node[line width=0mm, circle, fill, inner sep=1.5pt, label=right: $x_5$ ](x5){};
	\draw[rotate=108] (0,16)  node[line width=0mm, circle, fill, inner sep=1.5pt, label=left: $z$ ](z){};
	\draw[rotate=252] (0,16) node[line width=0mm, circle, fill, inner sep=1.5pt, label=right: $t$ ](t){};
	\draw[rotate=0] (0,18)  node[line width=0mm, circle, fill, inner sep=1.5pt, label=above: $u$ ](u){};
	\draw[rotate=180] (0,15)  node[line width=0mm, circle, fill, inner sep=1.5pt, label=below: $v$ ](v){};

	\draw [line width=0.6mm, blue] (w) -- (x1);
	\draw [line width=0.6mm, blue] (w) -- (x2);
	\draw [line width=0.6mm, blue] (w) -- (x3);
	\draw [line width=0.6mm, blue] (w) -- (x4);
	\draw [line width=0.6mm, blue] (w) -- (x5);	
	\draw [line width=0.6mm, blue] (x5) -- (x1);
	\draw [line width=0.6mm, blue] (x1) -- (x2);
	\draw [line width=0.6mm, blue] (x2) -- (x3);
	\draw [line width=0.6mm, blue] (x3) -- (x4);
	\draw [line width=0.6mm, blue] (x4) -- (x5);	
	\draw [dashed, line width=0.6mm, red] (z) -- (x2);
	\draw [line width=0.6mm, blue] (z) -- (x3);
	\draw [dashed, line width=0.6mm, red] (t) -- (x4);
	\draw [line width=0.6mm, blue] (t) -- (x5);
	\draw [bend left=12, line width=0.6mm, blue] (v) to (z);
	\draw [bend right=12, line width=0.6mm, blue] (v) to (t);
	\draw [line width=0.6mm, blue] (v) -- (x3);
	\draw [dashed, line width=0.6mm, red] (v) -- (x4);
	\draw [line width=0.6mm, blue] (u) -- (x1);
	\draw [line width=0.6mm, blue] (u) -- (x2);
	\draw [dashed, line width=0.6mm, red] (u) -- (x5);
	\draw [bend left=18, dashed, line width=0.6mm, red] (u) to (t);
	\draw [bend right=18, dashed, line width=0.6mm, red] (u) to (z);

	\draw[rotate=-15] (0,10)  node[circle, draw=black!0, inner sep=0mm, minimum size=0mm] (){\color{red} $-$};
	\draw[rotate=108] (0,10)  node[circle, draw=black!0, inner sep=0mm, minimum size=0mm] (){\color{red} $-$};
	\draw[rotate=180] (0,10)  node[circle, draw=black!0, inner sep=0mm, minimum size=0mm] (){\color{red} $-$};
	\draw[rotate=252] (0,10)  node[circle, draw=black!0, inner sep=0mm, minimum size=0mm] (){\color{red} $-$};
	\end{scope}
	
	\end{tikzpicture}
	\caption{A signed Wenger Graph}
	\label{fig:WengerGraph}
 \end{minipage}
\end{figure}

 \begin{definition}
 	Let $r$ be a positive real number. Let $\phi$ be a mapping of a set $\{v_1,v_2,\ldots, v_k\}$ of points (or vertices of a graph) to $C^r$. We denote by $I_{\phi; v_1,v_2,\ldots, v_{q}, \overbar{v}_{q+1},\ldots, \overbar{v}_k}$ an interval of minimum length which contains $\{\phi(v_i):  i=1,2,\ldots, q\} \cup \{\overbar{\phi(v_i)}: i=q+1, \ldots, k  \}$ and by $\ell_{\phi; v_1,v_2,\ldots, v_q, \overbar{v}_{q+1},\ldots, \overbar{v}_k}$ the length of this interval.
 \end{definition}
 
Note that the minimality of the length implies that the two end points of the interval $I_{\phi; v_1,v_2,\ldots, v_{q}, \overbar{v}_{q+1},\ldots, \overbar{v}_k}$ are in $\{\phi(v_i):  i=1,2,\ldots, q\}\cup \{\overbar{\phi(v_i)}: i=q+1, \ldots, k  \}$

\begin{lemma}\label{lem:mingadget}
Assume $\phi$ is a circular $(4+\alpha)$-coloring of the signed graph $(T, \pi)$ of Figure~\ref{fig:minigadget} with $0\leq \alpha < 2$. Then  $\ell_{\phi;x,y,z} \in [1-\frac{\alpha}{2}, 1+\frac{\alpha}{2}]$. Moreover, for any $t_1, t_2, t_3$ with $\max\{d_{\pod r}(t_i,t_j): i,j \in \{1,2,3\}\} \in [1-\frac{\alpha}{2}, 1+\frac{\alpha}{2}]$, there exists a circular $r$-coloring $\phi$ of $(T, \pi)$ such that $\phi(x)=t_1, \phi(y)=t_2, \phi(z)=t_3$.
\end{lemma}

\begin{proof}
Let $r=4+\alpha$ and let $\phi$ be a circular $r$-coloring of $(T, \pi)$. Without loss of generality we may assume that $\phi(x)$, $\phi(y)$ and $\phi(z)$ are on $C^r$ in the clockwise order, and assume the interval $[\phi(z), \phi(x)]$ is a   longest interval among $[\phi(x), \phi(y)]$, $[\phi(y), \phi(z)]$ and $[\phi(z), \phi(x)]$. Thus  $I_{\phi;x,y,z}=[\phi(x), \phi(z)]$.
We first claim that $[\phi(z), \phi(x)]$ contains $\overbar{\phi(y)}$. Otherwise, either $[ \overbar{\phi(y)}, \phi(y)]$ or $[\phi(y), \overbar{\phi(y)}]$ which is of length $\frac{r}{2}$, is included in either $(\phi(x), \phi(y)]$ or $[\phi(y), \phi(z))$. This is a contradiction as $[\phi(z), \phi(x)]$ is longest among the three. As $\overbar{\phi(y)}$ is contained in $[\phi(z), \phi(x)]$, and as $y$ is adjacent to both $z$ and $x$ with a negative edge, we conclude that $[\phi(z), \phi(x)]$ is of length at least 2. On the other hand, since $z$ and $x$ are adjacent with a negative edge, one of the two intervals, $[\phi(z), \phi(x)]$ or $[\phi(x), \phi(z)]$ is of length at most $\frac{r}{2}-1=1+\frac{\alpha}{2}$. As $\alpha<2$, the only option is that  $[\phi(x), \phi(z)]$ is of length at most $1+\frac{\alpha}{2}$.

For the other direction, assume $\ell_{\phi;x,y,z} < 1 - \frac{\alpha}{2}$, say $I_{\phi;x,y,z}=[0, \beta]$ for some $\beta < 1 - \frac{\alpha}{2}$. Each of $a,b,c$ is joined by a positive edge and a negative edge to vertices in $x,y,z$. This implies that $\phi(a), \phi(b), \phi(c) \in [1, 1+\beta+\frac{\alpha}{2}] \cup [3+\frac{\alpha}{2}, 3+\alpha+\beta]$. As each  of the intervals $[1, 1+\beta+\frac{\alpha}{2}]$ and $[3+\frac{\alpha}{2}, 3+\alpha+\beta]$ has length strictly smaller than $1$, two of the vertices $a,b,c$ are colored by colors of distance less than $1$ in $C^r$. But $abc$ is a triangle with three positive edges, a contradiction.

For the ``moreover" part, without loss of generality, we assume that $t_3=0, t_1\in [1-\frac{\alpha}{2}, 1+\frac{\alpha}{2}], t_2\in [0, t_1]$.   If $t_1\in [1-\frac{\alpha}{2}, 1]$, then let $\phi(a)=3+\frac{\alpha}{2}$, $\phi(b)=2$ and $\phi(c)=1$; if $t_1\in [1, 1+\frac{\alpha}{2}]$, then let $\phi(a)=3+\frac{\alpha}{2}$, $\phi(b)=2+\frac{\alpha}{2}$ and $\phi(c)=1$. It is straightforward to verify that $\phi$ is a circular $r$-coloring of $(T, \pi)$.
\end{proof}

By taking $\alpha=\frac{2}{3}-\epsilon$ and a switching at the vertex $z$, we have the following formulation of the lemma which we will use frequently.
 
\begin{corollary}\label{coro:mini-gadget}
Let $(T, \pi')$ be a signed graph obtained from $(T, \pi)$ by a switching at the vertex $z$, and let $\phi$ be a circular $(\frac{14}{3}-\epsilon)$-coloring of $(T, \pi')$ where $0< \epsilon<\frac{2}{3}$. Then $\ell_{\phi; x,y,\bar{z}} \in [\frac{2}{3}+\frac{\epsilon}{2},\frac{4}{3}-\frac{\epsilon}{2}]$.
\end{corollary}

We define $\tilde{W}$ to be the signed graph obtained from signed Wenger graph of Figure~\ref{fig:WengerGraph} by completing each of the four negative facial triangles to a switching of the mini-gadget of Figure~\ref{fig:minigadget}. Next we show that $\tilde{W}$ has a property similar to signed indicators, more precisely: 

\begin{lemma}\label{lem:gadget}
Let $r=\dfrac{14}{3}-\epsilon$ with $0<\epsilon\leq \dfrac{2}{3}$. For any circular $r$-coloring $\phi$ of $\tilde{W}$, $\ell_{\phi; u,v}\geq \dfrac{4}{9}$.
\end{lemma}

The proof of Lemma \ref{lem:gadget} is long, and we leave it to the next section. Let $\Gamma$ be obtained from $\tilde{W}$ by adding a negative edge $uv$. Let ${\mathcal I} = (\Gamma, u,v)$.  It follows from Lemma \ref{lem:gadget} that for $4 \le r < \frac {14}{3}$, $$({\mathcal I}, r) \subseteq [\dfrac 49, \dfrac r2 -1].$$

\begin{theorem}
	\label{thm:planar}
	Let $\Omega = K_4({\mathcal I})$. Then $\Omega$ is a signed planar simple graph with $\chi_c(\Omega) = \frac{14}{3}$.  
\end{theorem}
\begin{proof}
  First we show that $\Omega$ admits a circular $\frac{14}{3}$-coloring. 

	For $r = \frac{14}{3}$, there is a circular $r$-coloring $\phi$ of $\Gamma$ with $\phi(u)=\phi(v)$, defined as  $\phi(u)=\phi(v)=0$, $\phi(w)=3$, $\phi(x_1)=2$, $\phi(x_2)=1$, $\phi(x_3)=2$, $\phi(x_4)=\frac{1}{3}$, $\phi(x_5)=4$ and $\phi(z)=\phi(t)=1$. 
	 
	 We observe that each of the four negative triangles satisfies the conditions of Lemma~\ref{lem:mingadget}, and that the coloring of its vertices can be extended to the inner part of the mini-gadget. 
	 
	 Let $v_1,v_2,v_3,v_4$ be the 4 vertices of $K_4$.   
	 Then there is a circular $\frac{14}{3}$-coloring $\phi$ of $K_4({\cal I})$ with $\phi(v_i)= 0$ for $i=1,2,3,4$.  So $\chi_c(\Omega) \le  \frac{14}{3}$.   
	 
	 It remains to show that $\chi_c(\Omega) \ge  \frac{14}{3}$. Assume to the contrary that $\chi_c(\Omega)    <\frac{14}{3}$, let  $\phi$ be a circular $r$-coloring of $\Omega$ for some $4\le r < \frac{14}{3}$ (for the purpose of applying Lemma \ref{lem:gadget}, we assume $r \ge 4$).  Without loss of generality, assume $\phi(v_1)$, $\phi(v_2)$, $\phi(v_3)$ and $\phi(v_4)$ are on $C^r$ in this cyclic order. 
	  
	  As $({\mathcal I},r) \subseteq [ \frac 49, \frac r2 -1]$, we know that for any $1 \le i < j \le 4$, $$\frac 49 \le d_{\pod r}(\phi(v_i), \phi(v_j) ) \le \frac r2 -1.$$
	  By symmetry, we may assume $d_{\pod r}( \phi(v_1), \phi(v_3)) = \ell([\phi(v_1), \phi(v_3) ])$ and $d_{\pod r}( \phi(v_2), \phi(v_4)) = \ell([\phi(v_2), \phi(v_4) ])$. 
	  	
	  Hence   
	  $$\ell ([\phi(v_1), \phi(v_4) ]) = \ell([\phi(v_1), \phi(v_2) ]) +\ell([\phi(v_2), \phi(v_3) ]) +\ell([\phi(v_3), \phi(v_4) ])   \ge 3 \times \frac 49 = \frac 43 >  \frac r2 -1,$$
	  and 
	  $$\ell([\phi(v_4), \phi(v_1)]) \ge r -  (\ell([\phi(v_1), \phi(v_3) ]) +  \ell([\phi(v_2), \phi(v_4) ])) \ge 2   >\frac r2 -1.$$ 
	  This implies that $d_{\pod r}(\phi(v_1), \phi(v_4)) > \frac r2 -1$, a contradiction.  
\end{proof}

\section{Proof of Lemma \ref{lem:gadget}}

Assume to the contrary that $\phi$ is a circular $r$-coloring of $\tilde{W}$ with $\ell_{\phi; u,v} =\eta < \frac{4}{9}$. Without loss of generality, we   assume that $\phi(u)=0$ and $\phi(v)=\eta$.  Since each of $\phi(z)$ and $\phi(t)$ is of distance at least $1$ from both $\overbar{\phi(u)}$ and $\phi(v)$,  we have:

\begin{equation}\label{equ:z} \phi(z), \phi(t) \in \begin{cases} [1+\eta,\dfrac{4}{3}-\dfrac{\epsilon}{2}] \cup [\dfrac{10}{3}-\dfrac{\epsilon}{2}, \dfrac{11}{3}+\eta-\epsilon],  & \text{if } \eta \leq \dfrac{1}{3}-\dfrac{\epsilon}{2}, \cr  
  [\dfrac{10}{3}-\dfrac{\epsilon}{2}, \dfrac{11}{3}+\eta-\epsilon] ,  & \text{otherwise.}
  \end{cases}\end{equation}

\begin{lemma}   
\label{lem:w}
	 $\phi(w)\not\in (\dfrac{5}{3}-\epsilon, 3+\eta) \cup  (4-\dfrac{3\epsilon}{2}, \dfrac{2}{3}+\eta+\dfrac{\epsilon}{2})$.
\end{lemma}

\begin{proof}
Let $\phi(w)=\delta$. First we show that $\delta \not\in (\frac{5}{3}-\epsilon, 3+\eta)$. Assume to the contrary that $$ \delta\in (\frac{5}{3}-\epsilon, 3+\eta).$$ As $x_2$ is joined to $u$ and $w$ by positive edges,

\begin{equation}\label{equ:x_2}
	 \phi(x_2) \in  \begin{cases}
[1, \delta-1], & \text{ if } \delta >  \dfrac{8}{3}-\epsilon, \cr  
[\delta+1, \dfrac{11}{3}-\epsilon] & \text{ if   } \delta < 2, \cr
[1, \delta-1] \cup [\delta+1, \dfrac{11}{3}-\epsilon], & \text{ if } 2 \le  \delta \le  \dfrac{8}{3}-\epsilon.
\end{cases}
\end{equation}

\begin{equation}\label{equ:x_3} \phi(x_3) \in  \begin{cases}
[1+\eta, \delta-1] & \text{ if } \delta > \dfrac{8}{3}+\eta-\epsilon, \cr
[\delta+1, \dfrac{11}{3}+\eta-\epsilon] & \text{ if } \delta < 2+\eta, \cr
[1+\eta, \delta-1] \cup [\delta+1, \dfrac{11}{3}+\eta-\epsilon] & \text{ if } 2+ \eta \le \delta \le \dfrac{8}{3}+\eta-\epsilon.  
\end{cases}
\end{equation}

For a depiction of these cases, see Figure~\ref{fig:claim7.2}. 

 \begin{figure}[ht]
    \centering 
       \begin{tikzpicture}
	[scale=.5]
	\draw [line width=.05mm, black] (0,0) circle (4cm);
	\draw[rotate=0] (0,4)  node[line width=0mm, circle, fill, inner sep=1.2pt, label=below: \scriptsize{$0$} ] {};

         \draw [black, line width=.6mm, domain=-168:-228] plot ({3.8*cos(\x)}, {3.8*sin(\x)}) node [left] { \tiny{${\phi(z)}$} };
         \draw [black, line width=.6mm, domain=-14:-9] plot ({3.8*cos(\x)}, {3.8*sin(\x)}) node [left] { \tiny{${\phi(z)}$} };

	\draw [black, line width=.3mm, domain=-9:-66] plot ({3.6*cos(\x)}, {3.6*sin(\x)}) node [left] { \tiny{${\phi(x_3)}$} };
	\draw [black, line width=.3mm, domain=-140:-228] plot ({3.6*cos(\x)}, {3.6*sin(\x)}) node [right] { \tiny{${\phi(x_3)}$} };
	
	\draw [black, line width=.3mm, domain=12:-66] plot ({4.2*cos(\x)}, {4.2*sin(\x)}) node [right] { \tiny{$\phi(x_2)$} };
	\draw [black, line width=.3mm, domain=-140:-202] plot ({4.2*cos(\x)}, {4.2*sin(\x)}) node [left] { \tiny{$\phi(x_2)$} };
	\end{tikzpicture}
            \caption{A sketch of locating $\phi(z)$, $\phi(x_2)$ and $\phi(x_3)$ on $C^r$}
            \label{fig:claim7.2} 
\end{figure}
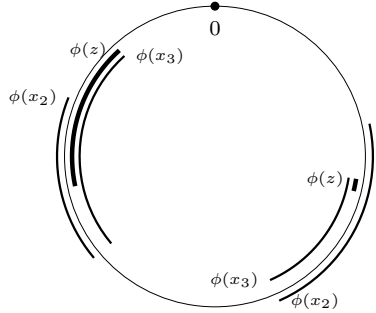

\begin{claim} \label{claim-phiz}
	The following restrictions on the value of $\phi(z)$ hold:
	  \begin{itemize}
	  	\item[I.] If $\delta < 3$ and $\phi(x_2) \in  [1, \delta-1]$, then $\eta \leq \frac{1}{3}-\frac{\epsilon}{2}$ and 
	  	$\phi(z)\in [1+\eta,\frac{4}{3}-\frac{\epsilon}{2}]$.  
	  	\item[II.] If $\phi(x_2) \in [\delta+1, \frac{11}{3}-\epsilon] $, then $ \phi(z) \in [\frac{10}{3}-\frac{\epsilon}{2}, \frac{11}{3}+\eta-\epsilon]$.  
	  	\item[III.] If $\phi(x_3)\in [1+\eta, \delta-1]$, then  $\phi(z)\in [\frac{10}{3}-\frac{\epsilon}{2}, \frac{11}{3}+\eta-\epsilon]$.  
	  	\item[IV.]  If $\delta> \frac{5}{3}+\eta-\epsilon$ and $\phi(x_3)\in [\delta+1, \frac{11}{3}+\eta-\epsilon]$, then $\eta \leq \frac{1}{3}-\frac{\epsilon}{2}$ and $\phi(z)\in [1+\eta, \frac{4}{3}-\frac{\epsilon}{2}]$. 
	  \end{itemize}
\end{claim}

\begin{proof*}
\medskip
[I]  Assume to the contrary (by \ref{equ:z}) that $ \phi(z) \in [\frac{10}{3}-\frac{\epsilon}{2}, \frac{11}{3}+\eta-\epsilon]$ and $\phi(x_2) \in  [1, \delta-1]$.  Then  $$d_{\pod r}( \phi(x_2), \phi(z)) \ge \min \{\frac{10}{3}-\frac{\epsilon}{2} - (\delta -1), \frac{14}{3}- \epsilon +1- (\frac{11}{3}+\eta-\epsilon ) \} > \frac{4}{3}-\frac{\epsilon}{2},$$   contradicting the fact that $x_2z$ is a negative edge.

\medskip
 [II]  Assume to the contrary (by \ref{equ:z}) that   $\phi(z) \in [1+\eta,\frac{4}{3}-\frac{\epsilon}{2}]$ and $\eta \le \frac 13 - \frac{\epsilon}{2}$.  
Then  $$d_{\pod r}( \phi(x_2), \phi(z)) \ge \min \{2+\eta, \delta -\frac{1}{3}+\frac{\epsilon}{2} \} > \frac{4}{3}-\frac{\epsilon}{2},$$   contradicting the fact that $x_2z$ is a negative edge.

\medskip
[III]  Assume to the contrary (by   \ref{equ:z}) that $\phi(z)  \in [1+\eta,\frac{4}{3}-\frac{\epsilon}{2}]$,  $\phi(x_3)\in [1+\eta, \delta-1]$ and hence $\delta \geq 2+\eta$. As $\delta \in (\frac{5}{3}-\epsilon, 3+\eta)$,  
 $$d_{\pod r}( \phi(x_3), \phi(z)) \le \delta-1 - (1+\eta) < 1,$$   contradicting the fact that $x_3z$ is a positive edge.

\medskip
[IV]   
Assume to the contrary (by \ref{equ:z}) that $\phi(z)\in [\frac{10}{3}-\frac{\epsilon}{2}, \frac{11}{3}+\eta-\epsilon]$. As  $\delta> \frac{5}{3}+\eta-\epsilon$,  
$$d_{\pod r}( \phi(x_3), \phi(z)) \le \frac{11}{3}+\eta-\epsilon - (\delta+1)   < 1,$$   contradicting the fact that $x_3z$ is a positive edge.

This completes the proof of Claim \ref{claim-phiz}.
 	\end{proof*}
 
To complete the proof of Lemma \ref{lem:w},   we partition the interval $(\frac{5}{3}-\epsilon, 3+\eta)$ into three parts and consider three cases depending on to which part $\delta$ belongs.

\medskip
\noindent
{\bf Case (i)}  $\delta \in (\dfrac{5}{3}-\epsilon, 2+\eta)$.

As   $\delta < 2+\eta$, by \ref{equ:x_3},  $\phi(x_3) \in [\delta+1, \frac{11}{3}+\eta-\epsilon]$. Thus $\overbar{\phi(x_3)}\in [\delta-\frac{4}{3}+\frac{\epsilon}{2}, \frac{4}{3}+\eta-\frac{\epsilon}{2}]$. 

By \ref{equ:x_2}, $\phi(x_2) \in [1, \delta-1] \cup [\delta+1, \frac{11}{3}-\epsilon]$. 

\medskip
\noindent
{\bf Subcase (i-1)} $\phi(x_2)\in [1, \delta-1]$ and hence  (by \ref{equ:x_2}) $\delta \ge 2$.

 \begin{figure}[ht]
    \centering 
       \begin{tikzpicture}
	[scale=.5]
	\draw [line width=.05mm, black] (0,0) circle (4cm);
	\draw[rotate=0] (0,4)  node[line width=0mm, circle, fill, inner sep=1.2pt, label=below: \scriptsize{$0$} ] {};

         \draw [black, line width=.6mm, domain=-9:-14] plot ({3.8*cos(\x)}, {3.8*sin(\x)}) node [left] { \tiny{${\phi(z)}$} };
	\draw [black, line width=.3mm, domain=-28:40] plot ({4.4*cos(\x)}, {4.4*sin(\x)}) node [right] { \tiny{${\overbar{\phi(x_3)}}$} };
	\draw [black, line width=.3mm, domain=-8:12] plot ({4.2*cos(\x)}, {4.2*sin(\x)}) node [left] { \tiny{$\phi(x_2)$} };
	\end{tikzpicture}
            \caption{Subcase (i-1): Restrictions on the negative triangle $x_3x_2z$.} 
            \label{fig:subcase(i-1)}
\end{figure}
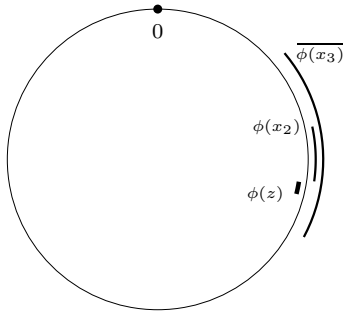

As $\delta < 2+ \eta < 3$, by [I],    $\phi(z)\in [1+\eta, \frac{4}{3}-\frac{\epsilon}{2}]$ and  $\eta\leq \frac{1}{3}-\frac{\epsilon}{2}$. Hence $\delta<2+\eta\leq \frac{7}{3}-\frac{\epsilon}{2}$. 

Consider the interval $I_{\phi; \bar{x}_3, x_2, z}$, see Figure~\ref{fig:subcase(i-1)}. If $\overbar{\phi(x_3)}$ is the starting point of this interval, then since  $\delta -1 < \frac{4}{3}-\frac{\epsilon}{2}$,  we have
$$[\overbar{\phi(x_3)}, \phi(z)]\subseteq [\delta-\frac{4}{3}+\frac{\epsilon}{2}, \frac{4}{3}-\frac{\epsilon}{2}].$$
If the starting of  $I_{\phi; \bar{x}_3, x_2, z}$ is   $\phi(x_2)$ or $\phi(z)$, then since $\delta -1 < \frac{4}{3}-\frac{\epsilon}{2}$, we have 
$$[\phi(x_2), \overbar{\phi(x_3)}]\subseteq [1, \frac{4}{3}+\eta-\frac{\epsilon}{2}].$$
In either case,   $I_{\phi; \bar{x}_3, x_2, z}$ has length at most $\frac{2}{3}-\epsilon$, contrary to Corollary \ref{coro:mini-gadget}. 

\medskip
\noindent
{\bf Subcase (i-2)} $\phi(x_2)\in [\delta+1, \frac{11}{3}-\epsilon]$. 

By  [II], $\phi(z)\in [\frac{10}{3}-\frac{\epsilon}{2}, \frac{11}{3}+\eta-\epsilon]$. 
Note that  $\ell([\frac{10}{3}-\frac{\epsilon}{2}, \frac{11}{3}+\eta-\epsilon]) =   \frac{1}{3}+\eta-\frac{\epsilon}{2}< 1$. 
Since $d_{\pod r}(\phi(x_3), \phi(z)) \ge 1$ (as $x_3z$ is a positive edge) and $ \phi(x_3)\in [\delta+1, \frac{11}{3}+\eta-\epsilon]$, we conclude that  $\delta \leq  \frac{5}{3}+\eta-\epsilon$ and $$ \phi(x_3)\in [\delta+1, \frac{8}{3}+\eta-\epsilon].$$
This implies that  $I_{\phi;x_3,x_2}\subseteq [\delta+1, \frac{11}{3}-\epsilon]$.
As $\delta > \frac 53 -\epsilon$,   $\ell([\delta+1, \frac{11}{3}-\epsilon]) < 1$, contrary to the fact that  $x_2x_3$ is a positive edge. 

\medskip
\noindent
{\bf Case (ii)} $\delta \in [2+\eta, \dfrac{8}{3}+\eta-\epsilon]$.

Depending on the ranges of $\phi(x_2)$ and $\phi(x_3)$, we consider four cases. 

\medskip
\noindent
{\bf Subcase (ii-1)} $\phi(x_2)\in [1, \delta-1]$ and $\phi(x_3)\in [1+\eta, \delta-1]$.

By [III], $\phi(z)\in [\frac{10}{3}-\frac{\epsilon}{2}, \frac{11}{3}+\eta-\epsilon]$. 
  
   As $\phi(x_2), \phi(x_3) \in [1, \delta-1]$, $\ell( [1+\eta, \delta-1]) <1$ and   $x_2x_3$ is a positive edge, we have $\delta \geq 3$ and $\phi(x_2)\in [1, \delta-2]$. However, the distance of points in $ [\frac{10}{3}-\frac{\epsilon}{2}, \frac{11}{3}+\eta-\epsilon]$ and $ [1, \delta-2]$ is at least $2-\eta$ which is strictly larger than $\frac{4}{3}-\frac{\epsilon}{2}$, contradicting that $x_2z$ is a negative edge.

\medskip
\noindent
{\bf Subcase (ii-2)} $\phi(x_2)\in [1, \delta-1]$ and $\phi(x_3)\in [\delta+1, \frac{11}{3}+\eta-\epsilon]$. ($\overbar{\phi(x_3)}\in [\delta-\frac{4}{3}+\frac{\epsilon}{2}, \frac{4}{3}+\eta-\frac{\epsilon}{2}]$)

%

By [IV],   $\phi(z)\in [1+\eta, \frac{4}{3}-\frac{\epsilon}{2}]$ and by \ref{equ:z}, $\eta \leq \frac{1}{3}-\frac{\epsilon}{2}$.

 Note that the interval  $I_{\phi;\bar{x}_3, x_2,z}$ is  one of the following intervals:
 \[ 
 [\overbar{\phi(x_3)}, \phi(z)]  \subseteq  [\delta-\frac{4}{3}+\frac{\epsilon}{2},  \frac{4}{3}-\frac{\epsilon}{2}], \ 
 [\overbar{\phi(x_3)}, \phi(x_2)]
 \subseteq  [\delta-\frac{4}{3}+\frac{\epsilon}{2}, \delta -1], \ [\phi(z), \phi(x_2)] \subseteq  [1+\eta, \delta-1],
 \]
\[ [\phi(z), \overbar{\phi(x_3)}]  \subseteq [1+\eta, \frac{4}{3}+\eta-\frac{\epsilon}{2}], \
[\phi(x_2), \overbar{\phi(x_3)}], [\phi(x_2), {\phi(z)}] \subseteq   [1, \frac{4}{3}+\eta-\frac{\epsilon}{2}].
\]  

 All the above intervals have lengths at most $\frac{2}{3}-\epsilon$, implying that $\ell_{\phi; \bar{x}_3, x_2, z} < \frac{2}{3}+\frac{\epsilon}{2}$, this contradicts Corollary~\ref{coro:mini-gadget}.

\medskip
\noindent
{\bf Subcase (ii-3)} 
$\phi(x_2)\in [\delta+1, \frac{11}{3}-\epsilon]$ and $\phi(x_3)\in [1+\eta, \delta-1]$.

%
%

By  \ref{equ:x_2}, $\delta\leq \frac{8}{3}-\epsilon$, and 
 by [III], $\phi(z)\in [\frac{10}{3}-\frac{\epsilon}{2}, \frac{11}{3}+\eta-\epsilon]$. 
Observe  that $\overbar{\phi(x_3)} \in [\frac{10}{3}+\eta-\frac{\epsilon}{2}, \frac{4}{3}+\delta-\frac{\epsilon}{2}]$.
So $I_{\phi; \bar{x}_3, x_2, z}$ is one of the following intervals:

$$[\overbar{\phi(x_3)}, \phi(z)], [\overbar{\phi(x_3)}, \phi(x_2)] \subseteq[\frac{10}{3}+\eta-\frac{\epsilon}{2},  \frac{11}{3}+\eta-\epsilon], \
[\phi(x_2), \overbar{\phi(x_3)}] \subseteq [\delta+1, \frac{4}{3}+\delta-\frac{\epsilon}{2}], $$
$$[\phi(x_2), {\phi(z)}]\subseteq [\delta+1,  \frac{11}{3}+\eta-\epsilon], 
\ [\phi(z), \phi(x_2)]\subseteq [\frac{10}{3}-\frac{\epsilon}{2}, \frac{11}{3}-\epsilon] \
\text{and} \ [\phi(z), \overbar{\phi(x_3)}] \subseteq [\frac{10}{3}-\frac{\epsilon}{2}, \frac{4}{3}+\delta-\frac{\epsilon}{2}].$$

Thus the $\ell_{\phi; \bar{x}_3, x_2, z} <\frac{2}{3}-\epsilon$, contradicting Corollary~\ref{coro:mini-gadget}.

 \medskip
 \noindent
 {\bf Subcase (ii-4)} 
$\phi(x_2) \in [\delta+1, \frac {11}{3}-\epsilon]$ and $\phi(x_3) \in [\delta+1, \frac{11}{3}+\eta - \epsilon]$.

The interval $[\delta+1, \frac{11}{3}+\eta-\epsilon]$  has length at most $\frac{2}{3}-\epsilon < 1$. This contradicts the fact that  $x_2x_3$ is a positive edge.
 
\medskip
\noindent
{\bf Case (iii)} $\delta \in (\dfrac{8}{3}+\eta-\epsilon, 3+\eta)$.

As $\delta>  \frac{8}{3}+\eta-\epsilon\geq  \frac{8}{3}-\epsilon$,  by \ref{equ:x_2} and \ref{equ:x_3}, $\phi(x_2)\in [1, \delta-1]$ and  $\phi(x_3)\in [1+\eta, \delta-1]$. 

As $\ell([1+\eta, \delta-1]) < 1$  and $d_{\pod r}(\phi(x_2),
\phi(x_3)) \ge 1$, we conclude that $\delta \geq 3$ and   $\phi(x_2)\in [1, \delta-2]$.
As $x_2z$ is a negative edge, and the distance between the intervals $[\frac{10}{3}-\frac{\epsilon}{2}, \frac{11}{3}+\eta-\epsilon]$ and $[1, \delta-2]$ is strictly larger than $\frac{4}{3}-\frac{\epsilon}{2}$, we know that $\phi(z) \notin [\frac{10}{3}-\frac{\epsilon}{2}, \frac{11}{3}+\eta-\epsilon]$. By \ref{equ:z}, $\phi(z)\in [1+\eta, \frac{4}{3}-\frac{\epsilon}{2}]$ and  $\eta \leq \frac{1}{3}-\frac{\epsilon}{2}$.
This implies that  $\phi(z)$ and $\phi(x_3)$ are both in $[1+\eta, \delta-1]$. However, $\delta <3+\eta$, so $\ell([1+\eta, \delta-1] ) < 1$, contradicting  the fact that $x_3z$ is a positive edge.

This completes the proof that $\phi(w)\notin (\frac{5}{3}-\epsilon, 3+\eta)$.

We observe that in this proof, vertex $x_1$ played no role. In other words, the conclusion   holds for the signed subgraph induced on $G\setminus x_1$. In this subgraph a switching at $U=\{w,x_2,x_3,x_4, x_5\}$ results in an isomorphic copy where $x_4$ and $x_5$ play the role of $x_2$ and $x_3$. Thus for the mapping  $\phi'$ defined as $\phi'(v)= \phi(v)$ for 
$v \in V(\tilde{W}) - U$ and $\phi'(v)= \overbar{\phi(v)}$ for $v \in U$, we have $\phi'(w) \notin (\frac{5}{3}-\epsilon, 3+\eta)$. Hence  $\phi(w) \not\in (4-\frac{3\epsilon}{2}, \frac{2}{3}+\eta+\frac{\epsilon}{2})$. 
\end{proof}

If $\eta> 1-\frac{3\epsilon}{2}$, then by the Lemma \ref{lem:w}, we have no choice for $\phi(w)$. Thus we assume in the rest of the proof that $\eta\leq 1-\frac{3\epsilon}{2}$ and   \[ \phi(w)\in [3+\eta, 4-\dfrac{3\epsilon}{2}]\cup [\dfrac{2}{3}+\eta+\dfrac{\epsilon}{2}, \dfrac{5}{3}-\epsilon].\]
 
 The two cases will be consider separately.

\medskip
\noindent
{\bf Case A.} $ \phi(w)\in [3+\eta, 4-\frac{3\epsilon}{2}]$.

As $ux_5$ is a negative edge and $\phi(u)=0$, we have
$\phi(x_5) \in [\frac{10}{3}-\frac{\epsilon} 2, \frac 43 - \frac {\epsilon}2 ]$.
As $\ell( [3+\eta, 4-\frac{3\epsilon}{2}]) < 1$,  and $x_5w$ is a positive edge,  we conclude that $3+\eta, \phi(w), \phi(x_5), \frac43 - \frac{\epsilon}2$ occur in this cyclic order. This implies that $$\phi(x_5)\in [4+\eta, \frac{4}{3}-\frac{\epsilon}{2}].$$ 

For $i=1,2,3,4$,  by considering the edges between
$x_i$ and $u,v,w$, similar arguments as above lead to the following restrictions on the value of $\phi(x_i)$: \[ \phi(x_1)\in [1, 3-\frac{3\epsilon}{2}], \ \phi(x_2)\in [1, 3-\frac{3\epsilon}{2}], \ \phi(x_3)\in [1+\eta, 3-\frac{3\epsilon}{2}], \ \phi(x_4)\in [4+\eta, \frac{4}{3}+\eta-\frac{\epsilon}{2}].  \] 

By \ref{equ:z}, based on the choices of $\phi(z)$ and $\phi(t)$, we consider four cases.

 \medskip
 \noindent
 {\bf Case A-1}  $\eta \le \dfrac 13 - \dfrac {\epsilon}{2}$ and  $\phi(z), \phi(t)\in [1+\eta,\dfrac{4}{3}-\dfrac{\epsilon}{2}]$. 
 
 \begin{figure}[ht]
  \centering 
    \begin{minipage}{.48\textwidth}
    \centering 
       \begin{tikzpicture}
	[scale=.5]
	\draw [line width=.05mm, black] (0,0) circle (4cm);
	\draw[rotate=0] (0,4)  node[line width=0mm, circle, fill, inner sep=1.2pt, label=above: \scriptsize{$0$} ] {};

\draw [black, line width=.6mm, domain=-9:-14] plot ({3.4*cos(\x)}, {3.4*sin(\x)}) node [left] { \tiny{${\phi(t), \phi(z)}$} };

	\draw [black, line width=.3mm, domain=64:120] plot ({3.6*cos(\x)}, {3.6*sin(\x)}) node [below] { \tiny{${\phi(x_5)}$} };
	\draw [gray, dotted, line width=.3mm, domain=-14:64] plot ({3.6*cos(\x)}, {3.6*sin(\x)}) node [below]{};
	
	\draw [black, line width=.3mm, domain=40:-36] plot ({3.8*cos(\x)}, {3.8*sin(\x)}) node [left] { \tiny{${\phi(x_4)}$} };
	\draw [gray, dotted, line width=.3mm, domain=120:40] plot ({3.8*cos(\x)}, {3.8*sin(\x)}) node [left]{};
	
	\draw [black, line width=.3mm, domain=-140:-78] plot ({4.2*cos(\x)}, {4.2*sin(\x)}) node [above left] { \tiny{${\phi(x_3)}$} };
	\draw [gray, dotted, line width=.3mm, domain=-78:-9] plot ({4.2*cos(\x)}, {4.2*sin(\x)}) node [left] {};
	
	\draw [black, line width=.3mm, domain=-62:13] plot ({4.6*cos(\x)}, {4.6*sin(\x)}) node [above] { \tiny{$\phi(x_2)$} };
	\draw [gray, dotted, line width=.3mm, domain=-62:-140] plot ({4.6*cos(\x)}, {4.6*sin(\x)}) node [right] {};
	
	\draw [black, line width=.3mm, domain=-65:-140] plot ({4.4*cos(\x)}, {4.4*sin(\x)}) node [left] { \tiny{$\phi(x_1)$} };
	\draw [gray, dotted, line width=.3mm, domain=-65:13] plot ({4.4*cos(\x)}, {4.4*sin(\x)}) node [below] {};
	\end{tikzpicture}
            \caption{Case A-1: Updating ranges of $\phi(x_i)$'s}
            \label{fig:caseA-1}
            \end{minipage}
             \begin{minipage}{.48\textwidth}
    \centering 
       \begin{tikzpicture}
	[scale=.5]
	\draw [line width=.06mm, black] (0,0) circle (4cm);
	 \draw[rotate=0] (0,4)  node[line width=0mm, circle, fill, inner sep=1.2pt, label=below: \scriptsize{$0$} ] {};
	 \draw [black, line width=.3mm, domain=198:138] plot ({4.2*cos(\x)}, {4.2*sin(\x)}) node [left] { \tiny{$\phi(w)$} };
         \draw [black, line width=.3mm, domain=120:64] plot ({4.2*cos(\x)}, {4.2*sin(\x)}) node [right] { \tiny{${\phi(x_5)}$} };
         \draw [black, line width=.3mm, domain=-65:-140] plot ({4.2*cos(\x)}, {4.2*sin(\x)}) node [left] { \tiny{$\phi(x_1)$} };
	\end{tikzpicture}
            \caption{Case A-1: Restrictions on $x_1x_5u$}
            \label{fig:caseA-1.2}
            \end{minipage}
\end{figure}

We will update the ranges of $\phi(x_i)$'s as depicted in Figure~\ref{fig:caseA-1}. In this figure the range of each $\phi(x_i)$ is shown as an interval partitioned to two parts. The full interval represents the  restriction we have started with. We then show that the dotted part of the interval is not available for $\phi(x_i)$, thus updating the range to the solid part of the interval.

As $\ell([1+\eta,\frac{4}{3}-\frac{\epsilon}{2}]) < 1$, $ \phi(x_3)\in [1+\eta, 3-\frac{3\epsilon}{2}]$ and $zx_3$ is a positive edge, the points $1+\eta, \phi(z), \phi(x_3), 3-\frac{3\epsilon}{2}$ occur in $C^r$ in this cyclic order. This implies that  $$\phi(x_3)\in [2+\eta, 3-\frac{3\epsilon}{2}].$$

As $\ell([2+\eta, 3-\frac{3\epsilon}{2}] ) < 1$, $\phi(x_2)\in [1, 3-\frac{3\epsilon}{2}]$ and $x_2x_3$ is a positive edge, the points $1, \phi(x_2), \phi(x_3), 3-\frac{3\epsilon}{2}$ occurs in $C^r$ in this cyclic order. This implies that    $$\phi(x_2)\in [1, 2-\frac{3\epsilon}{2}].$$

By considering the positive edges $x_5t$ and then $x_4x_5$, similar arguments show that $$\phi(x_5)\in [4+\eta, \frac{1}{3}-\frac{\epsilon}{2}] \text{ and } \phi(x_4)\in [\frac{1}{3}+\eta+\epsilon, \frac{4}{3}+\eta-\frac{\epsilon}{2}].$$

Considering the positive edge $x_1x_2$ and the range of $\phi(x_2)$ given above,  a similar argument shows that $$\phi(x_1)\in [2, 3-\frac{3\epsilon}{2}] \text{ and hence } \overbar{\phi(x_1)}\in [\frac{13}{3}-\frac{\epsilon}{2}, \frac{2}{3}-\epsilon].$$

Now consider the negative triangle $x_1x_5u$.  
If $I_{\phi;\bar{x}_1,x_5, u} = [\overbar{\phi(x_1)}, \phi(x_5)]$, then $I_{\phi;\bar{x}_1,x_5, u} \subseteq [\frac{13}{3}-\frac{\epsilon}{2}, \frac{1}{3}-\frac{\epsilon}{2}]$ but $\ell([\frac{13}{3}-\frac{\epsilon}{2}, \frac{1}{3}-\frac{\epsilon}{2}]) =\frac{2}{3}-\epsilon <\frac{2}{3}+\frac{\epsilon}{2}$, contrary to Corollary \ref{coro:mini-gadget}.

Also $\phi(u)=0$ cannot be an end point of the interval $I_{\phi;\bar{x}_1,x_5, u}$, as $0$ is at distance less than $\frac{2}{3}+\frac{\epsilon}{2}$ from each of the four end points of the intervals that are the ranges of $\overbar{\phi(x_1)}$ and $\phi(x_5)$.

Thus $I_{\phi;\bar{x}_1,x_5, u}= [\phi(x_5), \overbar{\phi(x_1)}]$. By  Corollary~\ref{coro:mini-gadget},   $\ell([\phi(x_5), \overbar{\phi(x_1)}]) \ge \frac{2}{3}+\frac{\epsilon}{2}$. Thus   $$\ell([\phi(x_1), \phi(x_5)]) =  \frac{r}{2}-\ell([\phi(x_5), \overbar{\phi(x_1)}]) \le  \frac{5}{3}-\epsilon.$$
As  $\phi(x_1)\in [2, 3-\frac{3\epsilon}{2}]$, $\phi(w)\in [3+\eta, 4-\frac{3\epsilon}{2}]$ and $\phi(x_5)\in [4+\eta, \frac{1}{3}-\frac{\epsilon}{2}]$, we conclude that $\phi(w) \in [\phi(x_1), \phi(x_5)]$ (see Figure~\ref{fig:caseA-1.2}).  Since $wx_1$ and $wx_5$ are positive edges,  we have 
$$2 \le \ell([\phi(x_1), \phi(w)]) + \ell([\phi(w), \phi(x_5)]) = \ell([\phi(x_1), \phi(x_5)]) \le  \frac{5}{3}-\epsilon,$$ a contradiction.

\medskip
\noindent
{\bf Case A-2} $\eta \le \dfrac 13 -\dfrac {\epsilon}{2}$, $\phi(z)\in [1+\eta,\dfrac{4}{3}-\dfrac{\epsilon}{2}]$  and $\phi(t)\in [\dfrac{10}{3}-\dfrac{\epsilon}{2}, \dfrac{11}{3}+\eta-\epsilon]$. 

\medskip
The proof is similar to the previous case. 
The  positive edge $zx_3$ and the negative edge $tx_4$ further restrict the ranges of  $\phi(x_3)$, $\phi(x_4)$. Then,   the new  ranges of $\phi(x_3)$ and $\phi(x_5)$, together with the positive edges $x_3x_2$ and $x_4x_5$ further restrict the range of $\phi(x_2)$, $\phi(x_5)$. As the computations are very similar to the previous case, we just list the conclusion  of this argument:
 \[ \phi(x_3)\in [2+\eta, 3-\frac{3\epsilon}{2}], \ \phi(x_2)\in [1, 2-\frac{3\epsilon}{2}], \ \phi(x_5)\in [\frac{1}{3}+\eta+\epsilon, \frac{4}{3}-\frac{\epsilon}{2}] \ \text{and}\  \phi(x_4)\in [4+\eta, \frac{1}{3}-\frac{\epsilon}{2}]. \]

Next we consider the negative triangle $vx_3x_4$.  As  
  \[ \overbar{\phi(x_3)}\in [\frac{13}{3}+\eta-\frac{\epsilon}{2}, \frac{2}{3}-\epsilon], \ \phi(x_4)\in [4+\eta, \frac{1}{3}-\frac{\epsilon}{2}] \ \text{and} \ \phi(v)=\eta.\]
Similar analysis as in the previous case shows that $I_{\phi; \bar{x}_3, x_4, v}=[\phi(x_4), \overbar{\phi(x_3)}]$ and $\ell( \phi(x_4), \overbar{\phi(x_3)}) \geq \frac{2}{3}+\frac{\epsilon}{2}$. This means that $\ell([\phi(x_3), \phi(x_4)]) <2$. A similar argument shows that  $\phi(w) \in [\phi(x_3), \phi(x_4)]$. As $x_3w, x_4w$ are positive edges, we have $$2 \le \ell([\phi(x_3), \phi(w)]) + \ell([\phi(w), \phi(x_4)]) = \ell([\phi(x_3), \phi(x_4)]) < 2,$$ a contradiction.

\medskip
\noindent
{\bf Case A-3} $\eta \le \dfrac {1}{3} -\dfrac {\epsilon}{2}$,  $\phi(z)\in [\dfrac{10}{3}-\dfrac{\epsilon}{2}, \dfrac{11}{3}+\eta-\epsilon]$  and $\phi(t)\in [1+\eta, \dfrac{4}{3}-\dfrac{\epsilon}{2}]$. 

 \begin{figure}[ht]
  \centering 
    \begin{minipage}{.48\textwidth}
    \centering 
       \begin{tikzpicture}
	[scale=.5]
	\draw [line width=.06mm, black] (0,0) circle (4cm);
	\draw[rotate=0] (0,4)  node[line width=0mm, circle, fill, inner sep=1.2pt, label=above: \scriptsize{$0$} ] {};

\draw [black, line width=.6mm, domain=-168:-218] plot ({4.2*cos(\x)}, {4.2*sin(\x)}) node [left] { \tiny{${\phi(z)}$} };
\draw [black, line width=.6mm, domain=-9:-14] plot ({3.4*cos(\x)}, {3.4*sin(\x)}) node [left] { \tiny{${\phi(t)}$} };

	\draw [black, line width=.3mm, domain=64:120] plot ({3.6*cos(\x)}, {3.6*sin(\x)}) node [below] { \tiny{${\phi(x_5)}$} };
	\draw [gray, dotted, line width=.3mm, domain=-14:64] plot ({3.6*cos(\x)}, {3.6*sin(\x)}) node [below]{};
	
	\draw [black, line width=.3mm, domain=40:-36] plot ({3.8*cos(\x)}, {3.8*sin(\x)}) node [left] { \tiny{${\phi(x_4)}$} };
	\draw [gray, dotted, line width=.3mm, domain=120:40] plot ({3.8*cos(\x)}, {3.8*sin(\x)}) node [left]{};
	
	\draw [black, line width=.3mm, domain=-61:-9] plot ({4.2*cos(\x)}, {4.2*sin(\x)}) node [right] { \tiny{${\phi(x_3)}$} };
	\draw [gray, dotted, line width=.3mm, domain=-140:-59] plot ({4.2*cos(\x)}, {4.2*sin(\x)}) node [right] {};
	
	\draw [black, line width=.3mm, domain=-69:-140] plot ({4.4*cos(\x)}, {4.4*sin(\x)}) node [left] { \tiny{$\phi(x_2)$} };
	\draw [gray, dotted, line width=.3mm, domain=-69:13] plot ({4.4*cos(\x)}, {4.4*sin(\x)}) node [right] {};
	\end{tikzpicture}
            \caption{Case A-3: Updating ranges of $\phi(x_i)$'s}
            \label{fig:caseA-3}
            \end{minipage}
             \begin{minipage}{.48\textwidth}
    \centering 
       \begin{tikzpicture}
	[scale=.5]
	\draw [line width=.06mm, black] (0,0) circle (4cm);
	 \draw[rotate=0] (0,4)  node[line width=0mm, circle, fill, inner sep=1.2pt, label=below: \scriptsize{$0$} ] {};
	 \draw [black, line width=.3mm, domain=138:198] plot ({4.2*cos(\x)}, {4.2*sin(\x)}) node [left] { \tiny{$I_w$} };
\draw [black, line width=.3mm, domain=64:120] plot ({4.2*cos(\x)}, {4.2*sin(\x)}) node [left] { \tiny{$I_5$} };
	\draw [black, line width=.3mm, domain=-36:40] plot ({4.2*cos(\x)}, {4.2*sin(\x)}) node [right] { \tiny{$I_4$} };
	\draw [black, line width=.3mm, domain=-61:-9] plot ({4.4*cos(\x)}, {4.4*sin(\x)}) node [right] { \tiny{$I_3$} };
	\draw [black, line width=.3mm, domain=-140:-69] plot ({4.2*cos(\x)}, {4.2*sin(\x)}) node [below] { \tiny{$I_2$} };
	\end{tikzpicture}
            \caption{Case A-3: Restrictions on $wx_5x_4x_3x_2$}
            \label{fig:caseA-3.2}
            \end{minipage}
\end{figure}

We will update the ranges of $\phi(x_2), \dots, \phi(x_5)$ as depicted in Figure~\ref{fig:caseA-3}.

Recall that $\phi(x_2)\in [1, 3-\frac{3\epsilon}{2}]$ and $\phi(x_5)\in [4+\eta, \frac{4}{3}-\frac{\epsilon}{2}]$. If $\phi(x_2) \in [1,2)$, then $d_{\pod r}( \phi(x_2), \phi(z)) > \frac{4}{3}- \frac{\epsilon}{2}$, contrary to the fact that  $x_2z$ is a negative edge. Thus  
$$\phi(x_2)\in [2, 3-\dfrac{3\epsilon}{2}]:=I_2.$$
If $ \phi(x_5)\in  (\frac{1}{3}-\frac{\epsilon}{2}, \frac{4}{3}-\frac{\epsilon}{2}]$, then  $d_{\pod r}( \phi(x_5), \phi(t)) < 1$, contrary to the fact that  $x_5t$ is a positive edge. Therefore $$ \phi(x_5)\in [4+\eta, \frac{1}{3}-\frac{\epsilon}{2}]:=I_5.$$

Note that $\ell(I_5 ) < 1$. As $\phi(x_4)\in [4+\eta, \frac{4}{3}+\eta-\frac{\epsilon}{2}]$ and  $d_{\pod r}(\phi(x_5), \phi(x_4)) \ge 1$ ($x_4x_5$ is a positive edge), we conclude that 
the points $4+\eta,  \phi(x_5), \phi(x_4), \frac{4}{3}+\eta-\frac{\epsilon}{2}$ occurs in $C^r$ in this cyclic order and  
$$\phi(x_4)\in [\dfrac{1}{3}+\eta+\epsilon, \dfrac{4}{3}+\eta-\dfrac{\epsilon}{2}]:=I_4.$$
Similarly, $\ell(I_2) < 1$, and $x_2x_3$ is a positive edge. Recall that $\phi(x_3)\in [1+\eta, 3-\frac{3\epsilon}{2}]$. Thus the points $2, \phi(x_2), \phi(x_3)$ occurs in $C^r$ in this cyclic order. Hence $$\phi(x_3)\in [1+\eta, 2-\dfrac{3\epsilon}{2}]:=I_3.$$

Finally recall that  $$\phi(w)\in [3+\eta, 4-\frac{3\epsilon}{2}]:=I_w.$$
The intervals $I_w$, $I_5$, $I_4$, $I_3$, $I_2$ are each of length less than 1,  and except for $I_3$ and $I_4$ there is no intersection among them (see Figure~\ref{fig:caseA-3.2}). 
Since $\ell(I_3) < 1$, we have $\phi(x_4) \notin I_3$ (because $x_3x_4$ is a positive edge). 
Thus the points  $\phi(w),\phi(x_5), \phi(x_4), \phi(x_3), \phi(x_2)$ occur in this cyclic order.  As $C^r$ is of length $\frac{14}{3}-\epsilon$, 
the colors of some two consecutive vertices of the $5$-cycle $wx_5x_4x_3x_2$ is less than 1, but all the edges of this cycle are positive. This is a contradiction.

\medskip
\noindent
{\bf Case A-4}  $\phi(z), \phi(t)\in [\dfrac{10}{3}-\dfrac{\epsilon}{2}, \dfrac{11}{3}+\eta-\epsilon]$.

\medskip
Similarly, we obtain that $$ \phi(x_2)\in [2, 3-\frac{3\epsilon}{2}]:=I_2, \ \ \phi(x_4)\in [4+\eta, \frac{1}{3}+\eta-\frac{\epsilon}{2}]:=I_4,$$ 
$$ \phi(x_5)\in [\frac{1}{3}+\eta+\epsilon, \frac{4}{3}-\frac{\epsilon}{2}]:=I_5, \ \ \phi(x_1)\in [\frac{4}{3}+\eta+\epsilon, 2-\frac{3\epsilon}{2}]:=I_1.$$
The points $\phi(w)$, $\phi(x_4)$, $\phi(x_5)$, $\phi(x_1)$ and $\phi(x_2)$ occur in $C^r$ in this cyclic order. As all the edges of the 5-cycle $wx_4x_5x_1x_2$ are positive, this is a contradiction.

\medskip
\noindent 
{\bf Case B.} $\phi(w)\in [\dfrac{2}{3}+\eta+\dfrac{\epsilon}{2}, \dfrac{5}{3}-\epsilon]$.

Similarly, by considering the edges between each of $x_i$'s and vertices $u, v, w$, we have that 
$$\phi(x_1), \phi(x_2)\in [\dfrac{5}{3}+\eta+\dfrac{\epsilon}{2}, \frac{11}{3}-\epsilon], 
\phi(x_3)\in [\dfrac{5}{3}+\eta+\dfrac{\epsilon}{2}, \frac{11}{3}+\eta-\epsilon], $$
$$\phi(x_4)\in [\dfrac{10}{3}+\eta-\dfrac{\epsilon}{2}, \dfrac{2}{3}-\epsilon]\ \text{and} \ \phi(x_5)\in [\dfrac{10}{3}-\dfrac{\epsilon}{2}, \dfrac{2}{3}-\epsilon].$$ 
Based on the choices of $\phi(z)$ and $\phi(t)$, we have four sub-cases to discuss.

\medskip
\noindent 
{\bf Case B-1 } $\phi(z), \phi(t)\in [\dfrac{10}{3}-\dfrac{\epsilon}{2}, \dfrac{11}{3}+\eta-\epsilon]$.

 \begin{figure}[ht]
  \centering 
    \begin{minipage}{.48\textwidth}
    \centering 
       \begin{tikzpicture}
	[scale=.5]
	\draw [line width=.05mm, black] (0,0) circle (4cm);
	\draw[rotate=0] (0,4)  node[line width=0mm, circle, fill, inner sep=1.2pt, label=above: \scriptsize{$0$} ] {};

        \draw [black, line width=.6mm, domain=-168:-218] plot ({4.2*cos(\x)}, {4.2*sin(\x)}) node [above left] { \tiny{${\phi(t),\phi(z)}$} };

	\draw [black, line width=.3mm, domain=94:38] plot ({3.6*cos(\x)}, {3.6*sin(\x)}) node [below] { \tiny{${\phi(x_5)}$} };
	\draw [gray, dotted, line width=.3mm, domain=94:192] plot ({3.6*cos(\x)}, {3.6*sin(\x)}) node [below]{};
	
	\draw [black, line width=.3mm, domain=112:172] plot ({3.8*cos(\x)}, {3.8*sin(\x)}) node [right] { \tiny{${\phi(x_4)}$} };
	\draw [gray, dotted, line width=.3mm, domain=38:112] plot ({3.8*cos(\x)}, {3.8*sin(\x)}) node [left]{};
	
	\draw [black, line width=.3mm, domain=-118:-58] plot ({4.8*cos(\x)}, {4.8*sin(\x)}) node [below] { \tiny{${\phi(x_3)}$} };
	\draw [gray, dotted, line width=.3mm, domain=-214:-118] plot ({4.8*cos(\x)}, {4.8*sin(\x)}) node [left] {};
	
	\draw [black, line width=.3mm, domain=-137:-196] plot ({4.6*cos(\x)}, {4.6*sin(\x)}) node [left] { \tiny{$\phi(x_2)$} };
	\draw [gray, dotted, line width=.3mm, domain=-137:-58] plot ({4.6*cos(\x)}, {4.6*sin(\x)}) node [right] {};
	
	\draw [black, line width=.3mm, domain=-118:-58] plot ({4.4*cos(\x)}, {4.4*sin(\x)}) node [right] { \tiny{$\phi(x_1)$} };
	\draw [gray, dotted, line width=.3mm, domain=-118:-196] plot ({4.4*cos(\x)}, {4.4*sin(\x)}) node [below] {};
	\end{tikzpicture}
            \caption{Case B-1: Updating ranges of $\phi(x_i)$'s}
            \label{fig:caseB-1}
            \end{minipage}
             \begin{minipage}{.48\textwidth}
    \centering 
       \begin{tikzpicture}
	[scale=.5]
	\draw [line width=.06mm, black] (0,0) circle (4cm);
	 \draw[rotate=0] (0,4)  node[line width=0mm, circle, fill, inner sep=1.2pt, label=below: \scriptsize{$0$} ] {};
	 \draw [black, line width=.3mm, domain=20:-40] plot ({4.2*cos(\x)}, {4.2*sin(\x)}) node [right] { \tiny{$\phi(w)$} };
         \draw [black, line width=.3mm, domain=94:38] plot ({4.2*cos(\x)}, {4.2*sin(\x)}) node [right] { \tiny{${\phi(x_5)}$} };
         \draw [black, line width=.3mm, domain=-58:-118] plot ({4.2*cos(\x)}, {4.2*sin(\x)}) node [left] { \tiny{$\phi(x_1)$} };
	\end{tikzpicture}
            \caption{Case B-1: Restrictions on $x_1x_5u$}
            \label{fig:caseB-1.2}
            \end{minipage}
\end{figure}

We will update the ranges of $\phi(x_i)$'s as depicted in Figure~\ref{fig:caseB-1}.

The positive edges $zx_3$ and $tx_5$ further restrict the ranges of $\phi(x_3)$ and $\phi(x_5)$. Then the new ranges of $\phi(x_3)$ and $\phi(x_5)$, through the positive edges $x_3x_2$ and $x_5x_4$, further restrict the ranges of $\phi(x_2)$ and $\phi(x_4)$. By similar computation as previous cases, we have 
$$\phi(x_3)\in [\dfrac{5}{3}+\eta+\dfrac{\epsilon}{2}, \dfrac{8}{3}-\epsilon], \phi(x_2)\in  [\dfrac{8}{3}+\eta+\dfrac{\epsilon}{2}, \frac{11}{3}-\epsilon], \phi(x_5)\in [\dfrac{13}{3}+\eta-\dfrac{\epsilon}{2}, \dfrac{2}{3}-\epsilon] \text{ and } \phi(x_4)\in [\dfrac{10}{3}+\eta-\dfrac{\epsilon}{2}, \dfrac{13}{3}-2\epsilon].$$ 
Considering the positive edge $x_1x_2$ and the range of $\phi(x_2)$ given above, we obtain that $$\phi(x_1)\in [\dfrac{5}{3}+\eta+\dfrac{\epsilon}{2}, \dfrac{8}{3}-\epsilon].$$
Next we consider the negative triangle $x_1x_5u$. As $\overbar{\phi(x_1)}\in [4+\eta, \frac{1}{3}-\frac{\epsilon}{2}]$, $\phi(x_5)\in [\frac{13}{3}+\eta-\frac{\epsilon}{2}, \frac{2}{3}-\epsilon]$ and $\phi(u)=0$, similar analysis shows that $I_{\phi;\bar{x}_1,x_5,u}=[\overbar{\phi(x_1)}, \phi(x_5)]$ and $\ell([\overbar{\phi(x_1)}, \phi(x_5)]) \geq \frac{2}{3}+\frac{\epsilon}{2}$. It implies that $\ell([\phi(x_5), \phi(x_1)])=\frac{5}{3}-\frac{\epsilon}{2}<2$. We observe that $\phi(w)\in [\phi(x_5), \phi(x_1)]$ (see Figure~\ref{fig:caseB-1.2}) and since $x_5w, x_1w$ are both positive edges, we have that 
\[2\leq \ell([\phi(x_5, w)])+\ell(\phi(w), \phi(x_1))=\ell([\phi(x_5), \phi(x_1)])<2,
\] a contradiction.

\medskip
\noindent 
{\bf Case B-2 } $\eta \le \dfrac 13 -\dfrac {\epsilon}{2}$,  $\phi(z)\in [\dfrac{10}{3}-\dfrac{\epsilon}{2}, \dfrac{11}{3}+\eta-\epsilon]$ and $\phi(t)\in [1+\eta, \dfrac{4}{3}-\dfrac{\epsilon}{2}]$.

\medskip
The positive edge $zx_3$ and the negative edge $tx_4$ further restrict the ranges of $\phi(x_3)$ and $\phi(x_4)$ respectively. Then the new ranges of $\phi(x_3)$ and $\phi(x_4)$, through the positive edges $x_3x_2$ and $x_4x_5$, further restrict the ranges of $\phi(x_2)$ and $\phi(x_5)$. By similar computation as previous cases, we have 
$$\phi(x_3)\in [\dfrac{5}{3}+\eta+\dfrac{\epsilon}{2}, \dfrac{8}{3}-\epsilon],\phi(x_2)\in [\dfrac{8}{3}+\eta+\dfrac{\epsilon}{2}, \frac{11}{3}-\epsilon],\phi(x_4)\in [\dfrac{13}{3}+\eta-\dfrac{\epsilon}{2}, \dfrac{2}{3}-\epsilon]  \text{ and } \phi(x_5)\in [\dfrac{10}{3}-\dfrac{\epsilon}{2}, \dfrac{13}{3}-2\epsilon].$$
Next we consider the negative triangle $x_3x_4v$. As $\overbar{\phi(x_3)}\in [4+\eta, \frac{1}{3}-\frac{\epsilon}{2}]$, $\phi(x_4)\in [\frac{13}{3}+\eta-\frac{\epsilon}{2}, \frac{2}{3}-\epsilon]$ and $\phi(v)=\eta$, similar analysis shows that $I_{\phi;\bar{x}_3,x_4,v}=[\overbar{\phi(x_3)}, \phi(x_4)]$ and $\ell([\overbar{\phi(x_3)},  \phi(x_4)]) \geq \frac{2}{3}+\frac{\epsilon}{2}$. This means that $\ell([\phi(x_4), \phi(x_3)])<2$. We observe that $\phi(w)\in [\phi(x_4), \phi(x_3)]$ and as $x_4w, x_3w$ are both positive edges, we have that \[2\leq \ell([\phi(x_4, w)])+\ell(\phi(w), \phi(x_3))=\ell([\phi(x_4), \phi(x_3)])<2,
\] a contradiction.

\medskip
\noindent 
{\bf Case B-3 } $\eta \le \dfrac {1}{3}- \dfrac {\epsilon}{2}$ and  $\phi(z), \phi(t)\in [1+\eta,\dfrac{4}{3}-\dfrac{\epsilon}{2}]$. 

 \begin{figure}[ht]
  \centering 
    \begin{minipage}{.48\textwidth}
    \centering 
       \begin{tikzpicture}
	[scale=.5]
        \draw [line width=.05mm, black] (0,0) circle (4cm);
	\draw[rotate=0] (0,4)  node[line width=0mm, circle, fill, inner sep=1.2pt, label=above: \scriptsize{$0$} ] {};

        \draw [black, line width=.6mm, domain=-9:-14] plot ({4.2*cos(\x)}, {4.2*sin(\x)}) node [right] { \tiny{${\phi(t),\phi(z)}$} };

	\draw [black, line width=.3mm, domain=114:192] plot ({3.6*cos(\x)}, {3.6*sin(\x)}) node [right] { \tiny{${\phi(x_5)}$} };
	\draw [gray, dotted, line width=.3mm, domain=114:38] plot ({3.6*cos(\x)}, {3.6*sin(\x)}) node [below]{};
	
	\draw [black, line width=.3mm, domain=94:38] plot ({3.8*cos(\x)}, {3.8*sin(\x)}) node [right] { \tiny{${\phi(x_4)}$} };
	\draw [gray, dotted, line width=.3mm, domain=94:172] plot ({3.8*cos(\x)}, {3.8*sin(\x)}) node [left]{};
	
	\draw [black, line width=.3mm, domain=-214:-137] plot ({4.6*cos(\x)}, {4.6*sin(\x)}) node [below] { \tiny{${\phi(x_3)}$} };
	\draw [gray, dotted, line width=.3mm, domain=-137:-58] plot ({4.6*cos(\x)}, {4.6*sin(\x)}) node [left] {};
	
	\draw [black, line width=.3mm, domain=-118:-58] plot ({4.4*cos(\x)}, {4.4*sin(\x)}) node [right] { \tiny{$\phi(x_2)$} };
	\draw [gray, dotted, line width=.3mm, domain=-118:-196] plot ({4.4*cos(\x)}, {4.4*sin(\x)}) node [right] {};
	
	\draw [black, line width=.3mm, domain=-167:-137] plot ({4.2*cos(\x)}, {4.2*sin(\x)}) node [right] { \tiny{$\phi(x_1)$} };
	\draw [gray, dotted, line width=.3mm, domain=-137:-58] plot ({4.2*cos(\x)}, {4.2*sin(\x)}) node [below] {};
	\draw [gray, dotted, line width=.3mm, domain=-167:-196] plot ({4.2*cos(\x)}, {4.2*sin(\x)}) node [below] {};
	\end{tikzpicture}
            \caption{Case B-3: Updating ranges of $\phi(x_i)$'s}
            \label{fig:caseB-3}
            \end{minipage}
             \begin{minipage}{.48\textwidth}
    \centering 
       \begin{tikzpicture}
	[scale=.5]
	\draw [line width=.06mm, black] (0,0) circle (4cm);
	 \draw[rotate=0] (0,4)  node[line width=0mm, circle, fill, inner sep=1.2pt, label=below: \scriptsize{$0$} ] {};
	\draw [black, line width=.3mm, domain=20:-40] plot ({4.2*cos(\x)}, {4.2*sin(\x)}) node [right] { \tiny{$I_w$} };
	\draw [black, line width=.3mm, domain=-58:-118] plot ({4.2*cos(\x)}, {4.2*sin(\x)}) node [left] { \tiny{$I_2$} };
	\draw [black, line width=.3mm, domain=-137:-167] plot ({4.2*cos(\x)}, {4.2*sin(\x)}) node [left] { \tiny{$I_1$} };
	\draw [black, line width=.3mm, domain=192:114] plot ({4.2*cos(\x)}, {4.2*sin(\x)}) node [above] { \tiny{$I_5$} };
	\draw [black, line width=.3mm, domain=94:38] plot ({4.2*cos(\x)}, {4.2*sin(\x)}) node [right] { \tiny{$I_4$} };
	\end{tikzpicture}
            \caption{Case B-3: Restrictions on $wx_5x_4x_3x_2$}
            \label{fig:caseB-3.2}
            \end{minipage}
\end{figure}

We will update the ranges of $\phi(x_i)$'s as depicted in Figure~\ref{fig:caseB-3}.

Recall that  $\phi(x_2)\in [\frac{5}{3}+\eta+\frac{\epsilon}{2}, \frac{11}{3}-\epsilon]$ and $\phi(x_4)\in [\frac{10}{3}+\eta-\frac{\epsilon}{2}, \frac{2}{3}-\epsilon]$. 

If $\phi(x_2) \in (\frac{8}{3}-\epsilon, \frac{11}{3}-\epsilon]$, then $d_{\pod r}( \phi(x_2), \phi(z)) > \frac{4}{3}- \frac{\epsilon}{2}$, contrary to the fact that $x_2z$ is a negative edge. Thus  $$\phi(x_2)\in [\dfrac{5}{3}+\eta+\dfrac{\epsilon}{2}, \dfrac{8}{3}-\epsilon]:=I_2.$$
If $ \phi(x_4)\in [\frac{10}{3}+\eta-\frac{\epsilon}{2}, \frac{13}{3}+\eta-\frac{\epsilon}{2})$, then  $d_{\pod r}( \phi(x_4), \phi(t)) > \frac{4}{3}- \frac{\epsilon}{2}$, contrary to the fact that  $x_4t$ is a negative edge. Therefore $$ \phi(x_4)\in [\dfrac{13}{3}+\eta-\dfrac{\epsilon}{2}, \dfrac{2}{3}-\epsilon]:=I_4.$$

Note that $\ell(I_4) < 1$. As $\phi(x_5)\in [\frac{10}{3}-\frac{\epsilon}{2}, \frac{2}{3}-\epsilon]$ and $d_{\pod r}(\phi(x_5), \phi(x_4)) \ge 1$ ($x_4x_5$ is a positive edge), we conclude that $$\phi(x_5)\in [\dfrac{10}{3}-\dfrac{\epsilon}{2}, \dfrac{13}{3}-2\epsilon]:=I_5.$$
Similarly, $\ell(I_2) < 1$, and $x_2x_3$ is a positive edge. Recall that $\phi(x_3)\in [\frac{5}{3}+\eta+\frac{\epsilon}{2}, \frac{11}{3}+\eta-\epsilon]$. Thus $$\phi(x_3)\in [\dfrac{8}{3}+\eta+\dfrac{\epsilon}{2}, \frac{11}{3}+\eta-\epsilon]:=I_3.$$

By the restriction from the positive edges $x_1x_2$ and $x_1x_5$ and the new range $I_2, I_5$, we have $$\phi(x_1)\in [\dfrac{8}{3}+\eta+\dfrac{\epsilon}{2}, \dfrac{10}{3}-2\epsilon]:=I_1.$$ 

Finally recall that  $$\phi(w)\in [\dfrac{2}{3}+\eta+\dfrac{\epsilon}{2}, \dfrac{5}{3}-\epsilon]:=I_w.$$
The intervals $I_w$, $I_2$, $I_1$, $I_5$, $I_4$ are each of length less than $1$,  and there is no intersection among them. (see Figure~\ref{fig:caseB-3.2})  As $C^r$ is of length $\frac{14}{3}-\epsilon$, the colors of some two consecutive vertices of the $5$-cycle $wx_2x_1x_5x_4$ is less than $1$, but all the  edges of this cycle are positive. That is a contradiction.

\medskip
\noindent 
{\bf Case B-4 } $\eta \le \dfrac 13 -\dfrac {\epsilon}{2}$, $\phi(z)\in [1+\eta,\dfrac{4}{3}-\dfrac{\epsilon}{2}]$  and $\phi(t)\in [\dfrac{10}{3}-\dfrac{\epsilon}{2}, \dfrac{11}{3}+\eta-\epsilon]$. 

\medskip
Similarly we obtain that 
$$\phi(x_2)\in [\dfrac{5}{3}+\eta+\dfrac{\epsilon}{2}, \dfrac{8}{3}-\epsilon]:=I_2, \phi(x_3)\in [\dfrac{8}{3}+\eta+\dfrac{\epsilon}{2}, \frac{11}{3}+\eta-\epsilon]:=I_3, $$ 
$$\phi(x_4)\in [\dfrac{10}{3}+\eta-\dfrac{\epsilon}{2}, \dfrac{13}{3}-2\epsilon]:=I_4 \text{ and } \phi(x_5)\in [\dfrac{13}{3}+\eta-\dfrac{\epsilon}{2}, \dfrac{2}{3}-\epsilon]:=I_5.$$ Recall that $$\phi(w)\in [\dfrac{2}{3}+\eta+\dfrac{\epsilon}{2}, \dfrac{5}{3}-\epsilon]:=I_w.$$

The intervals $I_w$, $I_2$, $I_3$, $I_4$, $I_5$ are each of length less than $1$, and except for $I_3$ and $I_4$ there is no intersection among them. As $\ell(I_4) < 1$, $\phi(x_3) \notin I_4$ (since $x_3x_4$ is a positive edge).
That is again a contradiction because of the  $5$-cycle $wx_2x_3x_4x_5$ all whose edges are positive.
 
 This completes the proof of Lemma \ref{lem:gadget}.

\section{  Questions and Remarks}\label{sec:Remarks}

A notion of a circular coloring of signed graphs was introduced in \cite{KS18}. It is different from the definition in this paper essentially because the concept of ``antipodal" points are defined differently. Both definitions use points on a circle as colors (the discrete version in \cite{KS18} uses $Z_k$ as colors, and we can view elements of $Z_k$ as points uniformly distributed on a circle). In \cite{KS18}, a fixed diameter of the circle is chosen, and the antipodal of a point is obtained by flipping the circle along the chosen diameter.  Thus for such a coloring, the colors are not symmetric. In particular, for each of the   two end points of the chosen diameter, its antipodal is itself.  In some sense, the definition in \cite{KS18} more faithfully extends the coloring of signed graphs that allows $0$ (as opposed to $0$-free coloring) introduced by Zaslavsky, where $0$ is a special color, whose antipodal is $0$ itself. We consider the speciality of a certain color  to be an undesirable feature. A  circular object should be invariant under rotation. In this sense, the circular coloring of signed graphs in this paper   more faithfully extends the circular coloring of graphs.

The circular coloring of graphs has been studied extensively in the literature. Many of the results and problems  on circular coloring of graphs would be interesting in the framework of signed graphs. We list some specific problems below and believe that there are many more interesting   problems.

\subsection{Jaeger-Zhang conjecture and extensions}\label{Jarger-ZhangExtension}

For a positive integer $k$, we have $\chi_c(C_{-2k})=\frac{4k}{2k-1}$. On the other hand, while for a negative odd cycle $C_{-(2k+1)}$ we have $\chi_c(C_{-(2k+1)})=2$, for the positive odd cycle $C_{+(2k+1)}$ we have we have $\chi_c(C_{+(2k+1)})=\chi_c(C_{2k+1})=\frac{2k+1}{k}$. These two facts can be stated uniformly by the following definition.

Given $ij \in \mathbb{Z}_2$, we say a closed walk $W$ of a signed graph $(G,\sigma)$ is of type $ij$ if the number of negative edges of $W$ (counting multiplicity) is congruent to $i \pmod{2}$, and 
the   total number of edges  (counting multiplicity) is congruent to   $j \pmod{2}$. For $ij \in \mathbb{Z}_2$ we define $g_{ij}(G, \sigma)$ to be the length of a shortest closed walk of type $ij$ in $(G,\sigma)$, setting it to be $\infty$ if there is no such a walk (see \cite{NSZ20} for corresponding no-homomorphism lemma and relation to coloring and homomorphism).

Let $C_l^{o+}$ be signed cycle of length $l$ where the number of positive edges is odd. Then   $\chi_c(C_l^{o+})=\dfrac{2l}{l-1}$. It is a well-known fact that a homomorphism of a graph onto an odd cycle gives an upper on its circular chromatic number.  
	The following theorem, whose proof we leave the the reader, is an extension of this fact.

\begin{theorem}
		Given a positive integer $l$ and a signed graph $(G, \sigma)$ satisfying $g_{ij}(G, \sigma)\geq g_{ij}(C_l^{o+})$ for $ij\in \mathbb{Z}^2_2$, we have $\chi_c(G, \sigma)\leq \dfrac{2l}{l-1}$ if and only if $(G, \sigma) \swto C_{l}^{o+}$.
\end{theorem}

The question of mapping planar graphs of odd girth large enough to $C_{2k+1}$ was shown by C.Q. Zhang (see \cite{KZ00} and \cite{Z02}) to be related to a conjecture of Jaeger in the theory of circular flow. A bipartite analogue of Jaeger-Zhang conjecture was introduced in \cite{NRS15} and studied in \cite{CNS20}, a first case of which is disproved in \cite{NPW20}. Thus we rather pose the following question:

\begin{problem}
Given a positive integer $l$, what is the smallest value $f(l)$ (with $f(\infty)=\infty$) such that for every signed planar graph $(G, \sigma)$ satisfying $g_{ij}(G, \sigma)\geq f(g_{ij}(C_l^{o+}))$ for all $ij\in \mathbb{Z}^2_2$, we have $\chi_c(G, \sigma)\leq \dfrac{2l}{l-1}$.
\end{problem}

That $f(3)=5$ is a restatement of the Gr\"otzsch theorem. That $f(4)=8$ is proved in \cite{NPW20}. For integers $l \geq 5$ it is known that $f(l)$ exists and is finite. Furthermore, $4k+1\leq f(2k+1)\leq 6k+1$ \cite{Z02, LTWZ13}, $f(5)\leq 11$ \cite{DP17} and $4k-2\leq f(2k)\leq 8k-2$ \cite{CNS20}.

\subsection{Hadwiger conjecture and extensions}\label{sec:minor}

One of the most intriguing conjectures in graph theory is the Hadwiger conjecture which tries to extend the four-color theorem. It claims that any graph without a $K_{k+1}$-minor is $k$-colorable. The case $k\leq 3$ of this conjecture is rather easy, but the case $k=4$ contains the four-color theorem. As the case $k+1$ would imply the case $k$, the difficulty of the conjecture only increases by $k$. Catlin \cite{Ca79} introduced a stronger version of the case $k=3$ which we restate below using the terminology of signed graphs and notion of circular coloring that we have introduced here. A signed graph $(H, \pi)$ is said to be a minor of $(G, \sigma)$ if it is obtained from $(G, \sigma)$ by a series of the following operations: 1. deleting vertices or edges, 2. contracting positive edge, 3. switching.

\begin{theorem}\label{thm:OddK_4}\cite{Ca79}
	If $(G,-)$ has no $(K_4, -)$-minor, then $\chi_c(G, +)\leq 3$. 
\end{theorem}

A possible strengthen of Catlin's result was proposed independently by B. Gerard and P. Seymour. This conjecture, which is stronger than the Hadwiger conjecture, is known as the Odd-Hadwiger conjecture and using the development in this work can be restated as follows.

\begin{conjecture}[Odd-Hadwiger]\label{conj:Odd-Hadwiger}
	If a signed graph $(G, -)$ has no $(K_{k+1}, -)$-minor, then $\chi_c(G, +)\leq k$. 
\end{conjecture}

To generalize this, one may ask:

\begin{problem}
\label{pro:oddminor}
	Assuming $(G, \sigma)$ has no $(K_{k+1}, -)$-minor, what is the best upper bound on $\chi_c(G, -\sigma)$?
\end{problem}

Observe that $\hat{K}^s_{2k}$ is the signed graph whose vertices are $1,2, \ldots k$ where each pair of distinct vertices are adjacent by both a negative edge and a positive edge, and each vertex has a negative loop. It follows from the structure of these signed graphs, in an edge-sign preserving mapping of a signed graph $(G,\sigma)$ to $\hat{K}^s_{2k}$, negative edges introduce no restriction, while vertices connected by a positive edge cannot be mapped to a same vertex. In other words, any such a mapping is a proper $k$-coloring of the subgraph $G_{\sigma}^+$ induced by the set of positive edges of $(G, \sigma)$. Recall that a switching homomorphism of $(G, \sigma)$ to $\hat{K}^s_{2k}$ is to find a signature $\sigma'$ equivalent to $\sigma$ and an edge-sign preserving homomorphism of $(G, \sigma')$ to $\hat{K}^s_{2k}$. Therefore, based on the following definition we have the next theorem. We define $$\chi_{+}(G, \sigma)=\min_{\sigma'\equiv\sigma}\{\chi(G_{\sigma'}^+)\}.$$

\begin{theorem}\label{thm:Circular-MaxX(G+)}
	Given a signed graph $(G, \sigma)$, we have $$2\chi_+(G, \sigma)-2< \chi_c(G, \sigma)\leq 2\chi_{+}(G, \sigma).$$
\end{theorem}

 Let $f(k)$ be the answer to Problem~\ref{pro:oddminor}. By Theorem~\ref{thm:Circular-MaxX(G+)} one observes that if Conjecture~\ref{conj:Odd-Hadwiger} holds, then $f(k)\leq 2k$. Similarly, considering the result of \cite{GGRSV09} we have $f(k)=O(n\sqrt{\log n})$.
 
\subsection{Signed planar graphs}

Let $D$ be the signed graph on two vertices $u$ and $v$ which are adjacent by two edges: one positive, another negative. This graph normally referred to as \emph{digon}. It is mentioned that $\chi_c(D)=4$, moreover, given $r\geq 4$, if $\phi$ is a circular $r$-coloring of $D$ where $\phi(u)=0$, then simply by the definition we have $\phi(v) \in (1, \frac{r}{2}-1)$. Thus, by Lemma~\ref{lem-indicator}, when $D$ is viewed as an indicator, we have $\chi_c(G(D))=2\chi_c(G)$ where $G$ is a graph (not signed) (this is a restatement of Corollary~\ref{cor-double}). 
In particular, we have $\chi_c(K_4(D))=8$. Noting this is a signed planar mulitgraph and that, by the four-color theorem, every signed planar multigraph without a loop admits an edge-sign preserving homomorphism to it, we obtain $\chi_c(\mathcal{SPM}))=8$ where $\mathcal{SPM}$ denotes the class of signed planar multigraphs. Furthermore, we recall that a signed graph with a positive loop admits no circular coloring and that adding a negative loop to a vertex of a signed graphs does not affect its circular chromatic number. 

For the class of signed planar simple graphs, the upper bound of $6$ follows from the fact that these graphs are $5$-degenerate. With our definition of circular chromatic number and development in this work, one may restate a conjecture of \cite{MRS16} as to ``circular chromatic number of the class of signed planar simple graphs is 4''. However, this conjecture is recently disproved in \cite{KN20}. The first counterexample provided in \cite{KN20} is essentially the subgraph $K_3(\mathcal{I})$ of the signed graph of Theorem~\ref{thm:planar} (they become a same signed graph after a switching). The work of \cite{KN20} is based on the dual interpretation of the circular four-coloring of signed planar graphs. The examples build there then are based on non-hamiltonian cubic bridgeless planar graphs. The underlying graph of the signed graph of Figure~\ref{fig:WengerGraph} is the dual of Tutte fragment used to build the first example of a non-hamiltonian cubic bridgeless planar graph and referred to as Wenger graph in some literature. This graph itself is used as a building block in a number of coloring results. Noting that a connection to a list coloring problem and circular $4$-coloring (of signed planar simple) graphs was established by the 3rd author, \cite{Z17}, we refer to \cite{KV18} for recent use of this gadget in refuting a similar conjecture.

We note, furthermore, that since in Theorem~\ref{thm:planar} we give the exact value of the circular chromatic number of $K_4(\mathcal{I})$, one does not expect to improve the lower bound using this particular gadget. 

It remains an open problem to decide the exact value of the circular chromatic number of the class of signed planar simple graphs or to improve the bounds (of $\frac{14}{3}$ and $6$) from either direction.

\subsection{Girth and planarity}

Some of the questions mentioned above can be generalized in the following way:

Given an integer $l$ and a class $\mathcal{C}$ of signed graphs, such as signed planar graphs or signed $K_4$-minor-free graphs, what is the circular chromatic number of signed graphs in $\mathcal{C}$ whose underlying graphs have girth $l$?

As an example, a result of \cite{CNS20} implies that every signed planar graph of girth at least $10$ admits a switching homomorphism to the signed graph $(K_4, e)$ which is the signed graph on $K_4$ with one negative edge. As this signed graph has circular chromatic number $3$, we conclude that 

\begin{theorem}
	For the class $\mathcal{SP}_{g \ge 10}$ of signed planar graphs $\tilde{G}$ of girth at least 10, we have $\chi_c(\mathcal{SP}_{g \ge 10})\leq 3$.
\end{theorem} 

We do not know if this bound is tight. 

In a more refined version of the question one might be given three values of $l_{01}$, $l_{10}$ and $l_{11}$ and be asked for a best bound on circular chromatic number of signed graphs in $\mathcal{C}$ which satisfy $g_{ij}(G, \sigma)\geq l_{ij}$.

\subsection{Spectrum}

In the previous question one may also be asked for the full possible range of circular chromatic number of a given family of signed graphs. For example it is known \cite{HZ}  that a rational number $r$ is   the circular chromatic number of a non-trivial   $K_4$-minor-free graph if and only if $r \in [2, \frac 83] \cup \{3\}$. As for signed $K_4$-minor-free simple graphs we extended the upper bound to $\frac{10}{3}$, it remains an open question whether each rational number between $\frac{8}{3}$ and $\frac{10}{3}$ is the circular chromatic number of a $K_4$-minor-free signed simple graph. Spectrum of the circular chromatic number of series-parallel graphs of given girth and circular chromatic number of planar graphs were studied in \cite{Moser1997, PZ2002,PZ2004,Zhu1999,Zhu1999b}. Similar questions are interesting for signed planar graphs and other families of signed graphs.

{\bf Acknowledgement} This work is supported by the ANR (France) project HOSIGRA (ANR-17-CE40-0022), by NSFC (China), Grant numbers: NSFC 11971438,
is also supported by the 111 project of The Ministry of Education of China. And it has received funding from the European Union's Horizon 2020 research and innovation programme under the Marie Sklodowska-Curie grant agreement No 754362. We would like to thanks Y. Jiang, L.A. Pham and T. Shan for discussions.

\end{document}